\documentclass[11pt,a4paper]{article}
\usepackage{fontenc}
\usepackage[latin1]{inputenc}
\usepackage{ucs}
\usepackage[french,english]{babel}
\usepackage{amssymb}
\usepackage{amsmath}
\usepackage{stmaryrd}
\usepackage{float}
\usepackage{theorem}
\usepackage{graphicx,epsfig}
\usepackage{verbatim}
\usepackage{bbm}

\usepackage{color}

\usepackage{multicol}
\usepackage{amssymb}
\usepackage{amsmath}
\usepackage{float}
\usepackage{theorem}
\usepackage{graphicx,epsfig}
\usepackage{verbatim}
\usepackage{bbm}

\newtheorem{theorem}{Theorem}[section]
\newtheorem{proposition}[theorem]{Proposition}
\newtheorem{lemma}[theorem]{Lemma}
\newtheorem{e-proposition}[theorem]{Proposition}
\newtheorem{corollary}[theorem]{Corollary}
\newtheorem{e-definition}[theorem]{Definition}
\newtheorem{remark}{\it Remark\/}

\newtheorem{definition}[theorem]{Definition}

\usepackage[official,right]{eurosym}

\newcommand{\ift}{\widehat{\cF}^{-1}}

\makeatletter
\DeclareRobustCommand{\qed}{%
  \ifmmode 
  \else \leavevmode\unskip\penalty9999 \hbox{}\nobreak\hfill
  \fi
  \quad\hbox{\qedsymbol}}
\newcommand{\openbox}{\leavevmode
  \hbox to.77778em{%
  \hfil\vrule
  \vbox to.675em{\hrule width.6em\vfil\hrule}%
  \vrule\hfil}}
\newcommand{\qedsymbol}{\openbox}
\newenvironment{proof}[1][Proof]{\par
  \normalfont
  \topsep6\p@\@plus6\p@ \trivlist
  \item[\hskip\labelsep\bfseries
    #1\@addpunct{.}]\ignorespaces
}{%
  \qed\endtrivlist
}
\makeatother

\newcommand{\id}[1]{\ensuremath{\mathbbm{1}_{#1}}}
\newcommand{\bc}{\begin{center}}
\newcommand{\ec}{\end{center}}
\newcommand{\1}{\mathbf{1}}

\newcommand{\E}{\mathbb{E}}

\renewcommand{\P}{\mathbb{P}}

\newcommand{\Q}{\mathbb{Q}}
\newcommand{\R}{\mathbb{R}}

\newcommand{\N}{\mathbb{N}}

\newcommand{\cC}{\mathcal{C}}

\newcommand{\cF}{\mathcal{F}}
\newcommand{\cL}{\mathcal{L}}

\newcommand{\cR}{\mathcal{R}}
\newcommand{\cH}{\mathcal{H}}

\newcommand{\pare}[1]{\left ( #1 \right )}
\newcommand{\croc}[1]{\left [ #1 \right ]}

\def \defi{:=}

\begin{document}

\title{Existence of a time inhomogeneous skew Brownian motion and some related laws}

\title{On the existence of a time inhomogeneous skew Brownian motion and some related laws}
\author{Pierre \'Etor\'e\thanks{Corresponding author, Laboratoire Jean Kuntzmann- Tour IRMA
51, rue des Math\'ematiques
38041 Grenoble Cedex 9, France. Phone: + 33 (0)4 76 63 56 93; E-mail: pierre.etore@imag.fr}  \and Miguel Martinez\thanks{Universit\'e Paris-Est Marne-la-Vall\'ee, Laboratoire d'Analyse et de
Math\'ematiques Appliqu\'ees, UMR $8050$, 5  Bld Descartes,
Champs-sur-marne, 77454 Marne-la-Vall\'ee Cedex 2, France.  Phone: +33 (0)1 60 95 75 25; E-mail: miguel.martinez@univ-mlv.fr} }
\date{\today}
\maketitle
\begin{center}
\strut \textbf{Abstract:}

\smallskip
\begin{minipage}{0.75\textwidth}
This article is devoted to the construction of a solution for the "skew inhomogeneous Brownian motion" equation:
\begin{equation*}
B^{\bf \beta}_t = x + W_t + \int_0^t {\bf \beta}(s)dL^0_s(B^{\bf \beta}),\hspace{0.4 cm}t\geq 0.
\end{equation*}
Here ${\bf \beta} : {\mathbb R}^+\rightarrow [-1,1]$ is a Borel function, $W$ is a standard Brownian motion, and $L^0_.(B^{\bf \beta})$ stands for the symmetric local time at $0$ of the unknown process $B^{\bf \beta}$. 

Using the description of the straddling excursion above a deterministic time $t$, we also compute the joint law of $\pare{B^{\bf \beta}_t, L^0_t(B^{\bf \beta}), G_t^\beta}$ where $G_t^\beta$ is the last passage time at $0$ before $t$ of $B^{\bf \beta}$.

\end{minipage}

\bigskip

\strut \textbf{Keywords:}

\smallskip
\begin{minipage}{0.75\textwidth}
Skew Brownian motion ; Local time ; Straddling excursion.
\end{minipage}

\end{center}

\section{Introduction}
\label{Intro}

\subsection{Presentation}
Consider $(W_t)_{t \ge 0}$ a standard Brownian motion on some filtered probability space
$(\Omega,\mathcal{F}, (\mathcal{F}_t)_{t \ge 0}, \mathbb{P})$  where the filtration satisfies the usual right continuity and completeness conditions.

Let us introduce $B^\beta$ the solution of
\begin{equation}
\label{ISBM}
B^{\bf \beta}_t = x + W_t + \int_0^t {\bf \beta}(s)dL^0_s(B^{\bf \beta}),\hspace{0.4 cm}t\geq 0
\end{equation}
where ${\bf \beta} : {\mathbb R}^+\rightarrow [-1,1]$ is a Borel function and $L^0_.(B^{\bf \beta})$ stands for the symmetric local time at $0$ of the unknown process $B^{\bf \beta}$.
The process $B^{\beta}$ will be called "time inhomogeneous skew Brownian motion" for reasons explained below. 

Of course, the equation (\ref{ISBM}) is an extension of the now well-known skew Brownian motion with constant parameter, namely the solution of (\ref{ISBM}) when the function $\beta$ is a constant in $(-1,1)$. 

The reader may find many references concerning the homogeneous skew Brownian motion and various extensions in the literature~:~
starting with the seminal paper by Harrisson and Shepp \cite{H-S}, let us cite \cite{Barlow-1}, \cite{Ouknine}, \cite{Portenko},  \cite{BKKM}, \cite{B-C-flow}, \cite{L-1}, \cite{R}, \cite{Z}, and the recent article \cite{ABTWW}. 
To complete the references on the subject, we mention the interesting survey by Lejay \cite{lejay-2006} and the cited articles therein.

On the contrary, concerning extensions of the skew Brownian motion in an inhomogeneous setting, we only found very few references~:~apart from the seminal paper by Weinryb \cite{Weinryb}, we mention \cite{BBKM}, where a variably skewed Brownian motion is constructed as the solution of a different equation than (\ref{ISBM}).

Up to our knowledge, concerning an existence result for possible solutions of equation (\ref{ISBM}) the situation reduces to only two references~:~we have already mentioned \cite{Weinryb} where it is said
that {\it ''partial existence results were obtained by Watanabe \cite{W-1, W-2}"} (to be precise, see however a detail explained in Remark \ref{rem-Weinryb}). Unfortunately, we have not been able to exploit these results fully in order to give a satisfactory response to the existence problem for the solutions of (\ref{ISBM}). So that, contrary to what is said in the introduction of \cite{BBKM}, the paper \cite{Weinryb} does not show that strong existence holds for equation (\ref{ISBM}) unless $\beta(s)\equiv \beta$ is a constant function.

Our second reference concerning the possible solutions of (\ref{ISBM}) is the fundamental book (posterior to \cite{Weinryb}) of Revuz and Yor \cite{RY} , Chapter VI  Exercise 2.24 p. 246, which starts with~: {\it "Let $B^\beta$ be a continuous semimartingale, \underline{\it if it exists}, such that (\ref{ISBM}) holds"} (in this quote, we adapted the notation to our setting ; we underlined what seems to be a crucial point).
\medskip
\medskip

In this article, we give the expected positive answer to the existence of weak solutions for equation (\ref{ISBM}) in the general case where the parameter function is a Borel function $\beta$ with values in $[-1,1]$. Our results may be completed with the result of Weinryb \cite{Weinryb}, where it is shown that pathwise uniqueness holds for equation (\ref{ISBM}). Then, the combination of both results ensures the existence of a unique strong solution to (\ref{ISBM}). 
\medskip

We will present essentially two ways of constructing a weak solution to (\ref{ISBM}).
\medskip

The first one is based on the description of the excursion that straddles some fixed deterministic time. Up to our knowledge, though the idea seems quite natural in our context, the recovery of the transition probability density of a skew Brownian motion (be it inhomogeneous or homogeneous) from the description of the excursion that straddles some fixed deterministic time seems to be new and is not mentioned in the survey paper \cite{lejay-2006}. As a by product, we also compute the trivariate density of the vector $(L_1^0\pare{B^{\bf \beta}}, G_1^\beta, B^{\bf \beta}_1)$ where $G_1^\beta:=\sup\{s<1~:~|B^{ \beta}_s| = 0\}$ is the last passage at $0$ before time $1$ of the constructed process. 
\medskip

Let us now explain briefly the second construction.

The main idea is to approximate the function $\beta$ by a monotone sequence of piecewise constant functions $\pare{\beta_n}_{n\geq 0}$. 
Still we have to face some difficulties and the construction, even in the simpler case of a given (fixed) piecewise constant coefficient $\bar{\beta}$, does not seem so trivial.

In order to treat the simpler case of a given piecewise constant coefficient $\bar{\beta}$, we are inspired by a construction for the classical skew Brownian motion with constant parameter which is explained in an exercise of the reference book \cite{RY}. This construction uses a kind of random flipping for excursions that come from an independent standard reflected Brownian motion $|B|$. The difficulty to adapt this construction in our inhomogeneous setting lies in the fact that it does not seem possible, at least directly and in order to construct $B^{\bar{\beta}}$, to combine it with "pasting trajectory" arguments at each point where $\bar{\beta}$ changes its value. This difficulty arises because the flow of a classical skew Brownian motion is not defined for all starting points $x$ simultaneously (see the remark in the introduction of \cite{B-C-flow} p.1694, just before Theorem 1.1). Still, we manage to adapt the "excursion flipping" arguments in our inhomogeneous setting and to identify our construction with a weak solution of equation (\ref{ISBM})~:~instead of trying to paste trajectories together, we show that our construction preserves the Markovian character of the reflected Brownian motion $|B|$. Then, the ideas developed in the previous sections allow us to show that the constructed process satisfies an equation of type (\ref{ISBM}), yielding a weak solution. The existence in the general case is then deduced by proving a strong convergence result.

\subsection{Organisation of the paper}

The paper is organised as follows :
\begin{itemize}
\item In Section 2 we first present results given in \cite{Weinryb} concerning the possible solutions of (\ref{ISBM}). We also recall some facts concerning the standard Brownian motion and its excursion straddling one. We expose in a separate subsection the result obtained in this paper.
\item In Section 3 we assume that we have a solution of (\ref{ISBM}) and we compute a one-dimensional marginal law of this solution using the inversion of the Fourier transform. Comparing these results with well-known results concerning the standard Brownian motion gives a hint on what should be the bivariate density of $(G_1^\beta, B^{\bf \beta}_1)$.
\item These hinted links are explained in Section 4~:~following the lines of \cite{BPY} for a more complicated but time-homogeneous process (namely Walsh's Brownian motion), we manage to give a precise description of what happens after  the last exit from $0$ before time $1$ for the solutions of (\ref{ISBM}). This permits to compute the Azema projection of $B^\beta$ on the filtration $({\cal F}_{G_t^\beta})$. In turn, this description enables us to retrieve the results of the previous section and to give a proof of the Markov property for $B^\beta$  in full generality. We finish this section by proving a Kolmogorov's continuity criterion, which is uniform w.r.t. the parameter function $\beta$. 
\item Using the results of the previous sections, we show the existence of solutions for the inhomogeneous skew Brownian equation (\ref{ISBM}) in Section 5. We give a first result of existence for the solutions of (\ref{ISBM}) in the case where $\beta$ is sufficiently smooth. In this case, the constructed solution is strong. The methods used in this part rely on stochastic calculus and an extension of the It\^ o-Tanaka formula in a time dependent setting due to \cite{peskir}. 

The general case is deduced by convergence and the Chapman-Kolmogorov equations.
\item As a by product of the study made in the preceding sections, we derive the joint distribution of $(L^0_1\pare{B^{\bf \beta}},G_1^\beta, B^{\bf \beta}_1)$ using well-known facts concerning the standard Brownian motion.
\item Finally, in the last section, we give a proof for the existence of the solution of (\ref{ISBM}), using a flipping of excursions argument and a convergence result. \end{itemize}

\section{Notations, preliminaries and main results}
\label{sec:main-results}

\subsection{Notations}

Throughout this note, ${\rm \bf e}$ denotes the exponential law of parameter $1$ ;\hspace{0.2 cm} $\text{\rm Arcsin}$ is the
standard arcsin law with density $(\pi\sqrt{y(1-y)})^{-1}$ on $[0,1]$;\hspace{0.2 cm} $\cR(p)$ denotes the Rademacher law with $p$ parameter {\it i.e.} the law of random variable $Y$ taking values $\{-1,+1\}$ with 
${\mathbb P}(Y=1)=p = 1 - {\mathbb P}(Y=-1)$.

We denote for all $t\geq 0$,
$$
G_t\defi\sup\{s<t~:~|W_s| = 0\}\quad\text{and}\quad G_t^\beta:=\sup\{s<t~:~|B^{ \beta}_s| = 0\}.
$$

It is well known that $G_1 \stackrel{\cL}{\sim}\text{\rm Arcsin}$ (see \cite{RY} Chap. III, Exercise 3.20).
We will also denote for all  $0\leq u\leq 1$,
$$
M_u\defi\big|W_{G_1+u(1-G_1)}\big|/\sqrt{1-G_1}   \quad\text{and}\quad  M^\beta_u\defi\big|B^\beta_{G^\beta_1+u(1-G^\beta_1)}\big|/\sqrt{1-G^\beta_1}.
$$
The process $M$ is called the {\it Brownian Meander} of length $1$. It is well known that $M_1\stackrel{\cL}{\sim}\sqrt{2{\rm \bf e}}$ (see See \cite{RY}, Chap. XII, Exercise 3.8).
\newline

{\bf Acronyms}~:~throughout the paper the acronym BM denotes a standard Brownian motion. The acronym SBM denotes a constant parameter skew Brownian motion solution of (\ref{ISBM}) for some constant function $\beta(s)\equiv \beta$. The acronym ISBM for Inhomogeneous Skew Brownian Motion denotes the solution of (\ref{ISBM}) in the case where $\beta$ is not a constant function. Unless there is no ambiguity, the character weak or strong of the considered solutions of (\ref{ISBM}) will be made precise.  

For a given semimartingale $X$, we denote by $L^0_.(X)$  its symmetric local time at level $0$.

The expectation $\mathbb{E}^{x}$ (resp. $\E^{s,x}$) refers to the probability measure $\mathbb{P}^{x}\defi\P(\,\cdot\,|B^\beta_0=x)$
 (resp. $\mathbb{P}^{s,x}\defi\P(\,\cdot\,|B^\beta_s=x)$).

\subsection{Preliminaries}
Let ${\beta} : {\mathbb R}^+\rightarrow [-1,1]$ a Borel function. 
\vspace{0.3cm}
The following fundamental facts are the key of many considerations of this paper. 
\begin{proposition}(see \cite{Weinryb} or \cite{RY} Chap. VI Exercise 2.24 p. 246)
\label{prop-fund-Weinryb}
Assume \eqref{ISBM} has a weak solution $B^\beta$.
Then under $\P^0$,
$$(|B_t^{ \beta }|)_{t\geq 0}  \stackrel{\cL}{\sim} \pare{|W_t|}_{t\geq 0}.$$  
\end{proposition}
We give the short proof for the sake of completeness.

\begin{proof}
Applying It\^o's formula we get on one side
$$
(W_t)^2=2\int_0^tW_sdW_s+t=2\int_0^t|W_s|\mathrm{sgn}(W_s)dW_s+t=2\int_0^t\sqrt{W_s^2}\,dZ_s+t,$$
where we have set $Z_t:=\int_0^t\mathrm{sgn}(W_s)dW_s$. Notice that $Z$ is a Brownian motion, thanks to Lévy's theorem.
On the other side we get
$$
(B^\beta_t)^2=2\int_0^tB^\beta_sdB^\beta_s+t=2\int_0^tB^\beta_sdW^\beta_s +t =2\int_0^t\sqrt{(B^\beta_s)^2}\,dZ^\beta_s+t,$$
where we have used $\id{B^\beta_s=0}dL^0_s(B^\beta)=dL^0_s(B^\beta)$, and where $Z^\beta_t:=\int_0^t\mathrm{sgn}(B^\beta_s)dW^\beta_s$, with $W^\beta$ the BM associated to the weak solution $B^\beta$. Notice that $Z^\beta$ is a Brownian motion, thus $(W)^2$ and $(B^\beta)^2$ are solutions of the same SDE, that enjoys uniqueness in law. This proves the result (see \cite{yamada1, yamada2}).
\end{proof}
\vspace{0.3cm}

\begin{theorem}(see \cite{Weinryb} or \cite{RY} Chap. VI Exercise 2.24 p. 246)
\label{theo-fund-Weinryb}

Pathwise  uniqueness holds for the weak solutions of equation (\ref{ISBM}).
\end{theorem}

\begin{remark}
\label{rem-Weinryb}
In the introductory article \cite{Weinryb}, it is shown that there is pathwise uniqueness for equation (\ref{ISBM}) but with a slight modification~: in \cite{Weinryb} the local time appearing in the equation is the standard right sided local time, so that the function ${\bf \beta}$ is supposed to take values in $(-\infty,1/2]$. Still, all the results of \cite{Weinryb} may be easily adapted for the case where $L_.^0(B^{\bf \beta})$ stands for the symmetric local time at $0$. We leave these technical aspects to the reader. 
\end{remark}

\vspace{0.3cm}

As $L_1^0(B^\beta)$, $M^\beta_1$ and $G^\beta_1$ (resp. $L_1^0(W)$, $M_1$ and $G_1$) are measurable functions of the trajectories of $|B^\beta|$ (resp. $|W|$), we get immediately the following corollary.

\begin{corollary}
\label{egal-loi-triplet}
We have
$$
(|B_1^{ \beta }|,L^0_1(B^\beta),G^\beta_1,M^\beta_1) \stackrel{\cL}{\sim} (|W_1|,L_1^0(W),G_1,M_1).
$$

\end{corollary}

The following known trivariate density will play a crucial role.

\begin{proposition}
\label{prop-exo-RY}
i) We have
\begin{equation}
\label{triplet}
(|W_1|,L_1^0(W),G_1)=(\sqrt{1-G_1}M_1,\sqrt{G_1}l^0,G_1),
\end{equation}
where $l^0\stackrel{\cL}{\sim}\sqrt{2{\rm \bf e}}$, and with $G_1,M_1,l^0$ independent.

ii) As a consequence, for all $t,s>0$, and all $\ell,x\geq 0$, the  image measure $\P^0[|W_t|\in dx,L_t^0(W)\in d\ell,G_t\in ds]$ is given by
\begin{equation}
\label{tridens-RY}
\1_{s\leq t}{\sqrt{\frac{2}{\pi s^3}}}\ell \exp\Big( -\frac{\ell^2}{2s} \Big)\frac{x}{\sqrt{2\pi(t-s)^3}}\exp\pare{-\frac{x^2}{2(t-s)}}\,ds\;d\ell\;dx.
\end{equation}
\end{proposition}

\begin{proof}
See \cite{RY}, Chap. XII, Exercise 3.8.
\end{proof}

\begin{remark}
\label{remark-dens-bro}
Note that, by integrating \eqref{tridens-RY} with respect to $\ell$, and using a symmetry argument we get that
\begin{equation}
\label{dens-bro}
p(t,0,y)=\frac{ |y|}{2\pi}\int_0^t\frac{1}{\sqrt{s}(t-s)^{3/2}}\exp\pare{-\frac{y^2}{2(t-s)}}ds,
\end{equation}
where $p(t,x,y) \defi \frac{1}{\sqrt{2\pi t}}\exp\big(-\frac{(y-x)^2}{2t}\big)$ is the transition density of a Brownian motion.
\end{remark}

\vspace{0.3cm}
\subsection{Transition probability density}

All through the paper the transition probability density of $B^\beta$ will be denoted $p^\beta(s,t ; x,y)$ (we show that it exists).

\vspace{0.2cm}
Let us now give the analytical expression of the function $p^\beta(s,t ; x,y)$. It will be shown later (Section~5) that $p^\beta(s,t ; x,y)$ is a transition probability function (in particular it satisfies the Chapman-Kolmogorov equations), and that the existing strong solution $B^\beta$ of 
\eqref{ISBM} is indeed an inhomogeneous Markov process with transition function $p^\beta(s,t ; x,y)$.

\vspace{0.2cm}
\begin{definition}
\label{def-1}
For all $t>0,y\in\R$, we set
\begin{equation}
\label{density-zero}
p^{{\bf \beta}}(0,t ; 0,y) \defi \frac{|y|}{\pi}\int_{0}^{t}\frac{1+\mathrm{sgn}(y){\bf \beta}(s)}{2}\frac{1}{\sqrt{s}(t-s)^{3/2}}\,{\rm exp}\pare{-\frac{y^2}{2(t-s)}}ds.
\end{equation}
\end{definition}

\vspace{0.2cm}

\begin{remark}
Note that, when ${\bf \beta}(s)\equiv \beta$ is constant, using \eqref{dens-bro} we have
$$
p^\beta(0,t;0,y)=(1+\beta)p(t,0,y)\id{y>0}+(1-\beta)p(t,0,y)\id{y<0}.$$
This is
the density of the SBM starting from zero with skewness parameter $\alpha:=(\beta+1)/2$, given for example in \cite{RY} Chap. III Exercise 1.16 p.87.
\end{remark}

Let us now introduce the shift operator $(\sigma_t)$ acting on time dependent functions as follows:
$${\bf \beta} \circ \sigma_t(s) = {\bf \beta}(t+s).$$ 
Assume for a moment that \eqref{ISBM} has a solution $B^\beta$ which enjoys the strong Markov property and satisfies
$$\P^0(B_t^\beta\in dy)=p^\beta(0,t;0,y)dy.$$

Let $x\neq 0$ be the starting point of $B^\beta$ at time $s$. Let 
$$T_0:=\inf(t\geq s~:~B_t^{\beta}=0).$$
 Since the local time $L^0_.(B^{\beta})$ does not increase until $B^{\beta}$ reaches $0$, the process $B^\beta$, heuristically speaking,
behaves like a Brownian motion on time interval $(s,T_0)$, implying that $\P^{s,x}(T_0\in du)=|x|\exp(-x^2/2(u-s))/\sqrt{2\pi (u-s)^3}$. Then it starts afresh from zero, behaving like an ISBM. Thus, for $t>s$,
\begin{equation}
\label{transition-density1}
\begin{array}{lll}
\P^{s,x}(B^\beta_t\in dy)&=&\displaystyle \P^{s,x}(B^\beta_t\in dy\,;\,s\leq T_0\leq t)+\P^{s,x}(B^\beta_t\in dy\,;\, T_0> t)\\
\\
&=&\displaystyle  dy\int_0^{t-s} \frac{|x| {\rm e}^{-{x^2}/{2u}} }{\sqrt{2\pi u^3}}p^{{\bf \beta}\circ \sigma_s \circ \sigma_u}(0, t-s-u ; 0,y)du\\
\\
&&\displaystyle \,+\,\frac{1}{\sqrt{2\pi (t-s)}}\croc{\exp\pare{-\frac{(y-x)^2}{2(t-s)}} - \exp\pare{-\frac{(y+x)^2}{2(t-s)}} }\id{xy>0}.
\end{array}
\end{equation}
The second line is a consequence of the assumed strong Markov property, while the third line is a consequence of the reflection principle due to the fact that on the event $\{T_0>t\}$ the process $B^\beta$ behaves like a Brownian motion.

But using \eqref{density-zero}, a Fubini-Tonelli argument, a change of variable, and \eqref{dens-bro}, we get~:
\begin{equation}
\label{petit-calcul}
\begin{array}{l}
\displaystyle \int_0^{t-s} \frac{|x| {\rm e}^{-{x^2}/{u}} }{\sqrt{2\pi u^3}}p^{{\bf \beta}\circ \sigma_s \circ \sigma_u}(0, t-s-u ; 0,y)du\\
\\
\displaystyle =\int_{u=0}^{t-s}\int_{r=0}^u
\frac{1+\mathrm{sgn}(y)\beta\circ\sigma_s(u)}{2}\sqrt{\frac{2}{\pi}}\frac{|y|}{(t-(s+u))^{3/2}}e^{-\frac{y^2}{2(t-(s+u))}}\frac{|x|}{2\pi\sqrt{r}(u-r)^{3/2}}e^{-\frac{x^2}{2(u-r)}}dr\,du\\
\\
\displaystyle =\int_0^{t-s}\frac{1+\mathrm{sgn}(y)\beta\circ\sigma_s(u)}{2}\frac{|y|}{\pi}\frac{e^{-\frac{y^2}{2(t-(s+u))}}}{\sqrt{u}(t-s-u)^{3/2}}e^{-x^2/2u}du.
\end{array}
\end{equation}

This leads us to the following definition.

\begin{definition}
\label{definition-2}
For $t>s$, $x,y\in\R$, we set
\begin{equation}
\label{transition-density-complete}
\begin{split}
&p^{{\bf \beta}}(s,t;x,y)\,\,:= \int_0^{t-s}\frac{1+\mathrm{sgn}(y)\beta\circ\sigma_s(u)}{2}\frac{|y|}{\pi}\frac{e^{-\frac{y^2}{2(t-(s+u))}}}{\sqrt{u}(t-s-u)^{3/2}}e^{-x^2/2u}du    \\
&\,+\,\frac{1}{\sqrt{2\pi (t-s)}}\croc{\exp\pare{-\frac{(y-x)^2}{2(t-s)}} - \exp\pare{-\frac{(y+x)^2}{2(t-s)}} }\id{xy>0}.
\end{split}
\end{equation}
\end{definition}

\vspace{0.7cm}

\begin{remark}
Note that in the case of Brownian motion ($\beta\equiv 0$) we have~:
\begin{equation}
\label{dens-bro2}
p(t,x,y)=\int_0^t  \frac{|y|}{2\pi}\frac{e^{-\frac{y^2}{2(t-u)}}e^{-\frac{x^2}{2u}}}{\sqrt{u}(t-u)^{3/2}} du+\,\frac{1}{\sqrt{2\pi t}}\croc{\exp\pare{-\frac{(y-x)^2}{2t}} - \exp\pare{-\frac{(y+x)^2}{2t}} }\id{xy>0}.
\end{equation}
Thus, considering  \eqref{transition-density-complete},
\begin{equation}
\label{transition-density-complete2}
p^\beta(s,t;x,y)=p(t-s,x,y)+\int_0^{t-s}\frac{\beta\circ\sigma_s(u)}{2}\frac{y}{\pi}\frac{e^{-\frac{y^2}{2(t-(s+u))}}}{\sqrt{u}(t-s-u)^{3/2}}e^{-x^2/2u}du. 
\end{equation}
This will be useful in forthcoming computations.
\end{remark}

\begin{remark}
When ${\bf \beta}(s)\equiv \beta$ is constant, $p^\beta(s,t;x,y)$ is just the transition density of the SBM given for example in \cite{RY}.
\end{remark}

\subsection{Main results}

We now state the main results obtained in this paper.

\begin{proposition}
\label{prop-densityzero}
Let $B^\beta$ a weak solution of \eqref{ISBM}.

For all $t>0,y\in\R$, we have
$$
\P^0(B^\beta_t\in dy) = p^\beta(0,t;0,y)dy,$$
where the function $p^\beta(0,t;0,y)$ is explicit in Definition \ref{def-1}.
\end{proposition}

The most important results of the paper may be summarized in the following theorem~:

\begin{theorem}
\label{existence-isbm}
Let ${\bf \beta} : {\mathbb R}^+\rightarrow [-1,1]$ a Borel function and $W$ a standard Brownian motion. For any fixed $x\in {\mathbb R}$, there exists a unique (strong) solution to \eqref{ISBM}. It is a (strong) Markov process with transition function $p^\beta(s,t;x,y)$ given by Definition \ref{definition-2}.
\end{theorem}

Still, a (very) little more work allows to retrieve the law of $(B^\beta_t,L^0_t(B^\beta),G^\beta_t)$ under $\P^0$.

\begin{theorem}
\label{theo-loi}
For all $t,s>0$, all $\ell\geq 0$ and all $x\in\R$, the  image measure $\P^0[ B^\beta_t\in dx,L_t^0(B^\beta)\in d\ell,G^\beta_t\in ds]$ is given by
\begin{equation}
\label{tridens-inhomo}
\1_{s\leq t}\frac{1+\mathrm{sgn}(x)\beta(s)}{2}\sqrt{\frac{2}{{\pi s^3}}}\,\ell\exp\Big( -\frac{\ell^2}{2s} \Big)\frac{|x|}{\sqrt{2\pi(t-s)^3}}\exp\pare{-\frac{x^2}{2(t-s)}}\,ds\;d\ell\;dx.
\end{equation}
\end{theorem}

\section{Law of the ISBM at a fixed time~:~proof of Proposition \ref{prop-densityzero}}
\label{sec:fourier}
 Let ${\bf \beta} : {\mathbb R}^+\rightarrow [-1,1]$ a Borel function. In this section the stochastic differential equation \eqref{ISBM} is assumed to have a weak solution $B^\beta$. It will be shown later on that this is indeed the case (see Theorem \ref{solution-faible}, Section \ref{sec-existence}).

\subsection{Proof of Proposition \ref{prop-densityzero} from the Fourier transform}
In this part, we note $g_{t,x}(\lambda):={\mathbb E}^{x} \exp \pare{i\lambda B_t^{{{\beta}}}}$ the Fourier transform of 
$B^\beta_t$ starting from $x$ and $h_x(t):={\mathbb E}^x \displaystyle \int_0^t {{\beta}}(s)dL^0_s(B^{{{\beta}}})$.

First, let us collect different results that will be used in the sequel to prove Proposition \ref{prop-densityzero}.

\vspace{0.2cm}

\begin{lemma}
\label{lemma-hzero}
We have 
$$h_0(t)= \frac{1}{\sqrt{2\pi}}\int_0^t \frac{{\beta}(s)}{\sqrt{s}}ds.$$
\end{lemma}

\begin{proof}
Using the symmetric Tanaka formula and Proposition \ref{prop-fund-Weinryb} we get
$\E^0(L_t^0(B^\beta))=\E^0|B^\beta_t|=\E^0|W_t|=\sqrt{\frac{2}{\pi}}\sqrt{t}$.

Consequently, we may apply Fubini's theorem and we get that,
\begin{equation*}
h_0(t)={\mathbb E}^0  \int_0^t {{\beta}}(s)dL^0_s(B^{{\beta}}) = \int_0^t {{\beta}}(s)d\pare{{\mathbb E}^0L^0_s(B^{{\beta}})}
 = \frac{1}{\sqrt{2\pi}}\int_0^t \frac{{\beta}(s)}{\sqrt{s}}ds.
\end{equation*}

\end{proof}

\begin{lemma}
\label{lemma-Fourier}
We have for all $\lambda >0$ and $t>0$,
$$
g_{t,0}(\lambda)={ e}^{-\lambda^2 t / 2}\pare{1+  \frac{i\,\lambda}{\sqrt{2\pi}}\int_0^t \frac{{\beta}(s)}{\sqrt{s}}{ e}^{\lambda^2 s / 2}ds}.
$$
\end{lemma}

\begin{proof}
Applying It\^o's formula ensures that for any fixed $\lambda\in\R$ the process $(g_{t,x}(\lambda))_{t\geq 0}$ is solution of the first order differential equation~:
$$
g_{t,x}(\lambda) = { e}^{i\lambda x} - \frac{\lambda^2}{2}\int_0^t g_{s}(\lambda)ds + i\lambda h_x(t),
$$
 (see \cite{Weinryb} or \cite{RY} Chap. VI Exercise 2.24 p. 246).
Solving formally this equation, we find that for any fixed $\lambda >0$~:
\begin{equation}
\label{eq:caracteristic}
g_{t,x}(\lambda)={ e}^{-\lambda^2 t / 2}\pare{{ e}^{i \lambda x} + i \lambda h_x(t){ e}^{\lambda^2 t / 2} - \frac{i\lambda ^3}{2}\int_0^t h_x(s){ e}^{\lambda^2 s / 2}ds}.
\end{equation}
Integrating by part, taking $x=0$ and using Lemma \ref{lemma-hzero}  we get the announced result.

\end{proof}

\begin{proof}[Proof of Proposition \ref{prop-densityzero}]
In the following computations we note $\ift(g)(z)\defi 2\pi\int_\R g(\lambda)e^{-iz\lambda}d\lambda$ the inverse Fourier transform of a function $g$.  We will sometimes write
$\ift(g(\lambda))(z)$ to make the dependence of $g$ with respect to $\lambda$ explicit.

 We have for $y\in\R$,
\begin{equation*}
\begin{array}{lll}
\displaystyle  p^\beta(0, t ; 0, y)&=&\ift(g_{t,0})(y)\\
\\
&=&\displaystyle \ift(e^{-\lambda^2t/2})(y)+2\pi\int_\R \frac{i\lambda}{\sqrt{2\pi}}\big(\int_0^t\frac{\beta(s)e^{-\lambda^2(t-s)/2}}{\sqrt{s}}ds\big)\,  e^{-iy\lambda}d\lambda\\
\\
&=&\displaystyle p(t,0,y) +  \frac{1}{\sqrt{2\pi}}\int_0^t\beta(s)2\pi\big(\int_\R i\lambda \frac{e^{(i\lambda)^2(t-s)/2}}{\sqrt{s}} e^{-iy\lambda}d\lambda \big)ds   \\
\\
&=&\displaystyle  p(t,0,y) + \frac{1}{\sqrt{2\pi}}\int_0^t\beta(s)\ift\big( i\lambda(t-s) \frac{e^{(i\lambda)^2(t-s)/2}}{\sqrt{s}(t-s)}   \big)(y)ds\\
\\
&=&\displaystyle p(t,0,y) + \frac{1}{\sqrt{2\pi}}\int_0^t\beta(s)\ift\big( \frac{d}{d\lambda}\big( \frac{e^{(i\lambda)^2(t-s)/2}}{\sqrt{s}(t-s)}\big)   \big)(y)ds\\
\\
&=&\displaystyle  p(t,0,y)+ \frac{1}{\sqrt{2\pi}}\int_0^ty\beta(s)\ift\big( \frac{e^{(i\lambda)^2(t-s)/2}}{\sqrt{s}(t-s)}   \big)(y)ds\\
\\
&=&\displaystyle  p(t,0,y)+ \frac{ y}{\sqrt{2\pi}}\int_0^t  \frac{\beta(s)}{\sqrt{s}(t-s)}\ift(e^{-\lambda^2(t-s)/2})(y)ds\\
\end{array}
\end{equation*}
so that
\begin{equation*}
\begin{array}{lll}
p^\beta(0, t ; 0, y)&=&\displaystyle p(t,0,y)+ \frac{ y}{\sqrt{2\pi}}\int_0^t  \frac{\beta(s)}{\sqrt{s}(t-s)}\frac{1}{\sqrt{2\pi(t-s)}}\exp\big(-\frac{y^2}{2(t-s)}\big)ds\\
\\
&=&\displaystyle p(t,0,y) +\ \frac{ y}{2\pi}\int_0^t\frac{\beta(s)}{\sqrt{s}(t-s)^{3/2}}\exp\big(-\frac{y^2}{2(t-s)}\big)ds.
\end{array}
\end{equation*}
Using \eqref{dens-bro}, we get the announced result.
\end{proof}

\subsection{Consequences of Proposition \ref{prop-densityzero}}

\vspace{0.7cm}

\begin{corollary}
\label{cor-loi}
We have, under ${\mathbb P}^0$,
\begin{equation}
B^{\bf \beta}_1 \stackrel{\cL}{\sim}Y\sqrt{1-G_1}\,M_1,
\end{equation}
where $G_1\stackrel{\cL}{\sim}\text{\rm Arcsin}$, $M_1\stackrel{\cL}{\sim}\sqrt{2{\rm \bf e}}$,
$G_1$ and $M_1$ are independent,
and where $Y$ denotes some r.v. independent of $M_1$ satisfying
$$
\cL\pare{Y\,|\,G_1=s}\stackrel{\cL}{\sim} \cR\pare{\frac{1+{\bf \beta}(s)}{2}}.
$$
\end{corollary}

\begin{proof}
First form the result of Corollary \ref{egal-loi-triplet}, we have that 
$$\pare{\sqrt{1-G^\beta_1} ,\,M^\beta_1 } \stackrel{\cL}{\sim}\pare{\sqrt{1-G_1}, M_1}$$
from which we retrieve that $G^\beta_1$ and $M^\beta_1$ are necessarily independent (see Proposition \ref{prop-exo-RY} for the independence between $G_1$ and $M_1$). Furthermore, using Proposition \ref{prop-densityzero} and easy computations of conditional expectations, we can see that under ${\P}^0$,
\begin{equation}
B^{\bf \beta}_1 \stackrel{\cL}{\sim}Y\sqrt{1-G_1}\,M_1,
\end{equation}
where $Y$ is  a random variable independent of $M_1$ satisfying
$$
\cL\pare{Y\,|\,G_1=s}\stackrel{\cL}{\sim} \cR\pare{\frac{1+{\bf \beta}(s)}{2}}.
$$
\end{proof}
 \begin{remark}
We have
 $$B_1^\beta=\mathrm{sgn}(B^\beta_1)\sqrt{1-G^\beta_1}\,M^\beta_1  \stackrel{\cL}{\sim}Y\sqrt{1-G_1}\,M_1$$
  with $G^\beta_1\stackrel{\cL}{\sim}G_1$
and $M^\beta_1\stackrel{\cL}{\sim}M_1$ and $Y$ constructed as above. Unfortunately, this is not enough to deduce the conditional law of $\mathrm{sgn}(B^\beta_1)$ w.r.t $(G_1^\beta, M_1^\beta)$. The result is completed in Proposition \ref{prop-loi-jointe-signe-g1-m1} below.
\end{remark}

\section{Last exit from $0$ before time $1$ and Markov property}
\label{sec-Azema-Markov}
Let us recall that in equation \eqref{ISBM}, we work with the symmetric ${\rm sgn}(.)$ function, satisfying ${\rm sgn}(0)=0$.

We now assume that $B^\beta$ is a strong solution of \eqref{ISBM} and that $\pare{{\cal F}_t}$ denotes the Brownian filtration of the Brownian motion $W$.

Recall also the definitions of 
\begin{itemize}
\item ${\cal F}_{G_1^\beta}$, the $\sigma$-algebra generated by the variables $H_{G_1^\beta}$, where $H$ ranges through all the $\pare{{\cal F}_t}$ optional (and thus predictable) processes (see \cite{RY}, Chap. XII p.488).
\item ${\cal F}_{G_1^\beta +}$, the $\sigma$-algebra generated by the variables $H_{G_1^\beta}$, where $H$ ranges through all the $\pare{{\cal F}_t}$ progressively measurable processes (see \cite{BPY}).
\end{itemize}

Throughout the section, all equalities involving conditional expectations have to be understood with the restriction that they hold only ${\mathbb P}$-almost surely. We will not precise it in our statements.
\subsection{Azema's projection of the ISBM}
\label{sec-azema}

\begin{proposition}
\label{prop-loi-jointe-signe-g1-m1}
We have, under ${\mathbb P}^0$,
\begin{equation}
\pare{{\rm sgn}(B^\beta_1), G_1^ \beta, M^\beta_1}\stackrel{\cL}{\sim}\pare{Y, G_1, M_1},
\end{equation}
where $G_1\stackrel{\cL}{\sim}\text{\rm Arcsin}$, $M_1\stackrel{\cL}{\sim}\sqrt{2{\rm \bf e}}$,
$G_1$ and $M_1$ are independent,
and where $Y$ denotes some r.v. independent of $M_1$ satisfying
$$
\cL\pare{Y\,|\,G_1=s}\stackrel{\cL}{\sim} \cR\pare{\frac{1+{\bf \beta}(s)}{2}}.
$$
moreover, in fact
\begin{equation}
\label{projection-signe}
{\mathbb E}^0\pare{{\rm sgn}(B^\beta_1)~|~{\cal F}_{G_1^\beta}} = \beta(G_1^\beta).
\end{equation}
\end{proposition}

\begin{remark}
 Notice that in particular ${\rm sgn}(B^\beta_1)$ is independent of $M^\beta_1$.

\end{remark}

\begin{proof}
Let $H$ denote an arbitrary real bounded $\pare{{\cal F}_s}$ predictable process.
The balayage formula implies on the one hand that
\begin{equation*}
\begin{split}
H_{G_t^\beta}\beta(G_t^\beta)|B^\beta_t | &= \int_0^t H_{G_u^\beta}\beta(G_u^\beta){\rm sgn}(B^\beta_{u})dW_u   \\
 &\;\;\;+ \int_0^t H_{G_u^\beta}\beta(G_u^\beta) dL^0_u(B^\beta).
\end{split}
\end{equation*}
On another hand it implies that 
\begin{equation*}
\begin{split}
H_{G_t^\beta}B^\beta_t &= \int_0^t H_{G_u^\beta}dB^\beta_u\\
&= \int_0^t H_{G_u^\beta}dW_u + \int_0^t H_{G_u^\beta}\beta(G_u^\beta)dL^0_u(B^\beta).
\end{split}
\end{equation*}
Making the difference, we see that
\begin{equation*}
\begin{split}
H_{G_t^\beta}B^\beta_t - H_{G_t^\beta}\beta(G_t^\beta)|B^\beta_t | &=  \int_0^t H_{G_u^\beta}dW_u - \int_0^t H_{G_u^\beta}\beta(G_u^\beta){\rm sgn}(B^\beta_{u})dW_u.
\end{split}
\end{equation*}

Thus, the process 
$$
\left \{H_{G_t^\beta} \pare{{\rm sgn}(B^\beta_t)- \beta(G_t^\beta)}|B^\beta_t | ~:t\geq 0~\right \}
$$
is a square integrable $({\cal F}_t)$ martingale.
In particular, we have that
$$
{\mathbb E}^0\pare{H_{G_t^\beta}{\rm sgn}(B^\beta_t)M^\beta_t\sqrt{t-G_t^\beta}}=  {\mathbb E}^0\pare{H_{G_t^\beta}\beta(G_t^\beta)\sqrt{\frac{\pi}{2}(t-G_t^\beta)}}.
$$
And since this equality is satisfied for all predictable process $H$,
\begin{equation}
\label{si-c-vrai-c-beau}
\begin{split}
{\mathbb E}^0\pare{{\rm sgn}(B^\beta_t)M^\beta_t~|~{\cal F}_{G_t^\beta}} &=\sqrt{\frac{\pi}{2}}\beta(G_t^\beta).
\end{split}
\end{equation}
This proves that ${\rm sgn}(B^\beta_t)$ and $M^\beta_t$ are conditionally  uncorrelated. However,  even though ${\rm sgn}(B^\beta_t)$ takes only values in $\{-1,1\}$ $\P^0$-a.s., this equality is not enough to deduce the conditional law ${\cal L}\pare{{\rm sgn}(B^\beta_t)~|~\sigma(G_t^\beta)}$ and we have to work a little more. 
In the following, we follow the lines of the article \cite{BPY} p.290.

Let $\pare{{\cal H}_t}$ the smallest right-continuous enlargement of $\pare{{\cal F}_t}$ such that $G_1^\beta$ becomes a stopping time. Then, according to 
Jeulin \cite{Jeulin} p.77 and the exchange formula, we have
\begin{equation}
\label{eq:germ}
{\cal H}_{G_1^\beta} = {\cal F}_{G_1^\beta +} = \sigma\pare{{\cal F}_{G_1^\beta}} \vee \bigcap_{n\geq {\mathbb N}^\ast} \sigma\pare{W_{G_1^\beta + u}~:~0\leq u\leq\frac{1}{n}}.
\end{equation}
Define for $\varepsilon \in (0,1)$, $G_1^{\beta,{\varepsilon}} = G^ \beta_{1} + \varepsilon (1- G^ \beta_{1})$~;~this is a family of $\pare{{\cal H}_t}$ stopping times, such that~:~${\cal H}_{G_1^{\beta,{\varepsilon}}}={\cal F}_{G_1^{\beta,{\varepsilon}}}$ (see again \cite{Jeulin}). Moreover, since $\pare{{\cal H}_t}$ is right-continuous, we have~:~
$$
{\cal F}_{G_1^\beta +} = {\cal H}_{G_1^\beta } = \bigcap_{\varepsilon \in (0,1)}{\cal F}_{G_1^{\beta,{\varepsilon}}}.
$$
We now proceed to show that $M_1^\beta$ is independent from ${\cal H}_{G_1^\beta}$. We first remark that the $\pare{{\cal F}_t}$ submartingale ${\mathbb P}\pare{G_1^\beta < t~|~{\cal F}_t}$ (for $t<1$) can be computed explicitly using the Theorem \ref{prop-fund-Weinryb}. We easily find that
$$
{\mathbb P}\pare{G_1^\beta < t~|~{\cal F}_t} = \Phi\pare{\frac{|B_t^\beta|}{\sqrt{1-t}}}
$$
where $\Phi(y):=\displaystyle \sqrt{\frac{2}{\pi}}\int_0^y\exp\pare{-\frac{x^2}{2}}dx$. We deduce from this, using the explicit enlargement formulae that~:~
\begin{equation}
\label{eq:sde-f-grossie}
\begin{split}
&\pare{|B^\beta_{G_1^\beta+u}| -  L^0_{G_1^\beta + u}(B^\beta)} - \pare{|B^\beta_{G_1^\beta}| - L^0_{G_1^\beta}(B^\beta)}\\
&= \vartheta_u + \int_0^u \frac{ds}{\sqrt{1- (G_1^\beta + s)}}\pare{\frac{\Phi '}{\Phi}}\pare{\frac{|B_{G_1^\beta + s}^\beta|}{\sqrt{1-(G_1^\beta + s)}}},\hspace*{0,4 cm}\text{for}\;\;u<1-G_1^\beta,
\end{split}
\end{equation}
where $\{\vartheta_u~:~u\geq 0\}$ is a $\pare{{\cal H}_{G_1^\beta +u},\,u\geq 0}$ Brownian motion, so that  $\{\vartheta_u~:~u\geq 0\}$ is independent from
${\cal H}_{G_1^\beta}$.

Note that $B^\beta_{G_1^\beta}=0$ and $\displaystyle  L^0_{G_1^\beta + u}(B^\beta) =  L^0_{G_1^\beta}(B^\beta)$ for $0\leq u<1-G_1^\beta$.  

Using Brownian scaling, we deduce that
\begin{equation}
\begin{split}
m^\beta_v & = \gamma_v + \int_0^v \frac{dh}{\sqrt{1- h}}\pare{\frac{\Phi '}{\Phi}}\pare{\frac{m_{h}^\beta}{\sqrt{1-h}}}\hspace*{0,4 cm}\text{for}\;\;v<1,
\end{split}
\end{equation}
where $\gamma_v := \frac{1}{\sqrt{1-G_1^\beta}}\vartheta_{(1-G_1^\beta)v}$ is again a Brownian motion which is independent from ${\cal H}_{G_1^\beta}$
and $m^\beta_v := \frac{|B^\beta_{G_1^\beta + v(1-G_1^\beta)}|}{\sqrt{1-G_1^\beta}}$. 

From this, we deduce that $\{m^\beta_v~:~v< 1\}$ is the unique strong solution of a SDE driven $\pare{\gamma_v}$. Consequently, $\{m^\beta_v~:~v< 1\}$ is independent of ${\cal H}_{G_1^\beta}$ and by continuity of $(m^\beta_v)_{0\leq v\leq 1}$ so is
$m^\beta_1 := M_1^\beta$.

From the fact that $B^\beta$ is a $\pare{{\cal F}_t}$ predictable process and (\ref{eq:germ}) (\ref{eq:sde-f-grossie}), we deduce that 
$$\bigcap_{n\geq {\mathbb N}^\ast} \sigma\pare{B^\beta_{G_1^\beta + u}~:~0\leq u\leq\frac{1}{n}}\subseteq {\cal F}_{G_1^\beta +}$$ 
and thus, since ${\rm sgn}(B^\beta_1) = {\rm sgn}(B^\beta_{G_1^\beta + 1/n})$ for all $n> 0$, the random variable  ${\rm sgn}(B^\beta_1)$ is ${\cal F}_{G_1^\beta +}$ measurable. So that,
\begin{equation*}
\label{si-c-vrai-c-beau-2}
\begin{split}
{\mathbb E}^0\pare{{\rm sgn}(B^\beta_1)M^\beta_1~|~{\cal F}_{G_1^\beta}} &={\mathbb E}^0\pare{{\mathbb E}^0\pare{{\rm sgn}(B^\beta_1)M^\beta_1~|~{\cal F}_{G_1^\beta +}}~|~{\cal F}_{G_1^\beta}}\\
&={\mathbb E}^0\pare{M^\beta_1}{\mathbb E}^0\pare{{\rm sgn}(B^\beta_1)~|~{\cal F}_{G_1^\beta}}\\
&=\sqrt{\frac{\pi}{2}}{\mathbb E}^0\pare{{\rm sgn}(B^\beta_1)~|~{\cal F}_{G_1^\beta}},
\end{split}
\end{equation*}
and identifying with (\ref{si-c-vrai-c-beau}) ensures that
$$
{\mathbb E}^0\pare{{\rm sgn}(B^\beta_1)~|~{\cal F}_{G_1^\beta}} = \beta(G_1^\beta) \pare{={\mathbb E}^0\pare{{\rm sgn}(B^\beta_1)~|~\sigma(G_1^\beta)}}.
$$
\end{proof}
\begin{remark}
\label{t1-norole}
The time $t=1$ plays no role in the above reasoning so the relation ${\mathbb E}^0\pare{{\rm sgn}(B^\beta_t)~|~{\cal F}_{G_t^\beta}} = 
\beta(G_t^\beta) $ holds also for any time $t$. This proves that, up to a modification, the dual predictable projection of the process 
$({\rm sgn}(B^\beta_t))_{t\geq 0}$ on the  filtration $({\cal F}_{G_t^\beta})$ is given by the process $(\beta(G_t^\beta))_{t\geq 0}$. 

This means that the fundamental equation of the Inhomogeneous Skew Brownian motion may be re-interpreted like forcing with $\beta$ a prescribed $({\cal F}_{G^ \beta_t})$-predictable projection for $({\rm sgn} (B^\beta_t))_{t\geq 0}$ in the following equation
\begin{equation}
\left \{
\begin{array}{l}
B^\beta_t = W_t + \int_0^t {}^{\rm p}\pare{{\rm sgn} (B^\beta_s)}dL_s^0(B^\beta)\\
{}^{\rm p}\pare{{\rm sgn} (B^\beta_t)} = \beta(G^\beta_t),
\end{array}
\right .
\end{equation}
where ${}^{\rm p}\pare{Y_.}$ is a notation for the $({\cal F}_{G^ \beta_t})$ predictable projection of the measurable process $Y$.
\end{remark}


\subsection{Markov property}
\label{sec-markov}

Using the results of the previous section, we may show that the inhomogeneous skew Brownian motion $B^\beta$ is a Markov process.
Indeed, even in the homogeneous case, the Markov property for the existing solution $B^\beta$, is up to our knowledge a non trivial 
question (see \cite{zait1,zait2,kulik,hajri,burdzy,maire-champagnat-10}).

\begin{proposition}
\label{prop-markov-process}
Let $f$ be a positive Borel function, then
\begin{equation*}
{\mathbb E}^x \pare{f(B^\beta_t) | {\cal F}_s} = \int_{-\infty}^\infty dy f(y) p^{\beta}(s,t ; B^\beta_s, y).
\end{equation*}
\end{proposition}
\begin{proof}

In the following computations, we will use various times the Fubini-Tonelli theorem, which is justified since we are dealing with positive integrable integrands. 

\begin{equation*}
\begin{split}
{\mathbb E}^x \pare{f(B^\beta_t) | {\cal F}_s}&={\mathbb E}^x \pare{f(B^\beta_t)\id{G^\beta_t \leq s} | {\cal F}_s} + {\mathbb E}^x \pare{f(B^\beta_t)\id{G^\beta_t > s} | {\cal F}_s}.
\end{split}
\end{equation*}

Let $H_s$ denote some ${\cal F}_s$ measurable random variable.
We have~:
\begin{equation*}
\begin{split}
&{\mathbb E}^x \pare{H_s\, f(B^\beta_t)\id{G^\beta_t > s}}\,={\mathbb E}^x \croc{H_s \,f\pare{{\rm sgn}\pare{B^\beta_t}\frac{|B^\beta_t|}{\sqrt{t-G^\beta_t}}\sqrt{t-G^\beta_t}}\id{G^\beta_t > s}}\\
&={\mathbb E}^x \croc{{\mathbb E}^x\croc{H_s \,f\pare{{\rm sgn}\pare{B^\beta_t}\frac{|B^\beta_t|}{\sqrt{t-G^\beta_t}}\sqrt{t-G^\beta_t}}\id{G^\beta_t > s} | {\cal F}_{G_t^\beta}}}.
\end{split}
\end{equation*}
Since the process $\pare{H_s\,\id{u > s}}_{u\geq 0}$ is $\pare{{\cal F}_u}$ predictable for any fixed time $s$, the random variable $H_s\id{G^\beta_t > s}$ is ${\cal F}_{G_t^\beta}$ measurable by definition of ${\cal F}_{G_t^\beta}$. So that
\begin{equation*}
\begin{split}
&{\mathbb E}^x \pare{H_s\, f(B^\beta_t)\id{G^\beta_t > s}}\,={\mathbb E}^x \croc{H_s\id{G^\beta_t > s} \,{\mathbb E}^x\croc{f\pare{{\rm sgn}\pare{B^\beta_t}\frac{|B^\beta_t|}{\sqrt{t-G^\beta_t}}\sqrt{t-G^\beta_t}} | {\cal F}_{G_t^\beta}}}\\
&={\mathbb E}^x \croc{H_s\id{G^\beta_t > s}\sum_{\delta\in \{ -1,1\}}\frac{1+\delta\beta(G_t^\beta)}{2}\int_{-\infty}^\infty f\pare{\delta\sqrt{t-G_t^\beta}y}\mu^0(dy)}
\end{split}
\end{equation*}
where we have used the fact that $\frac{|B^\beta_t|}{\sqrt{t-G^\beta_t}}$ is independent of ${\cal F}_{G_t^\beta}$ and ${\rm sgn}(B^\beta_t)$, and $\mu^0(dy)$ stands for the law of $\sqrt{2 {\rm \bf e}}$ (see Proposition \ref{prop-loi-jointe-signe-g1-m1}).

Note that the process $\left\{H_s K^s_u(f):=H_s\id{u > s}\sum_{\delta\in \{ -1,1\}}\frac{1+\delta\beta(u)}{2}f\pare{\delta\sqrt{t-u}y}~:~u\geq 0\right\}$ is $\pare{{\cal F}_u}$ predictable.  Since $G_t^\beta$ is the last exist time before time $t$ of some  reflected Brownian motion, well-known results concerning the dual predictable projection of last exit times for BM (and hence for reflected BM as well) ensure that~:~
$$
{\mathbb E}^x\pare{H_s K^s_{G^\beta_t}(f)} = {\mathbb E}^x\pare{\int_0^t H_s K^s_{u}(f)\pare{\frac{\pi}{2}(t-u)}^{-1/2}dL_u^{0}(|B^\beta|)}.
$$
In particular,  using Theorem \ref{prop-fund-Weinryb}, we have
\begin{equation*}
\begin{split}
&{\mathbb E}^x \pare{H_s\, f(B^\beta_t)\id{G^\beta_t > s}}\,=\int_0^\infty \mu^0(dy) {\mathbb E}^x \croc{\int_0^t H_s K^s_u(f)\pare{\frac{\pi}{2}(t-u)}^{-1/2}dL_u^{0}(|B^\beta|) }\\
&=\int_0^\infty \mu^0(dy){\mathbb E}^{x} \croc{H_s \int_0^t K^s_{u}(f)\pare{\frac{\pi}{2}(t-u)}^{-1/2} dL_u^{0}(|B^\beta|)}.
\end{split}
\end{equation*}
So that using the Markov property for the absolute value of Brownian motion, and denoting $\tilde{W}$ some standard Brownian motion independent of ${\cal F}_s$, we obtain~:
\begin{equation*}
\begin{split}
&{\mathbb E}^x \pare{H_s\, f(B^\beta_t)\id{G^\beta_t > s}}\\
&=\int_0^\infty \mu^0(dy) {\mathbb E}^{x} \croc{ H_s {\mathbb E}^{|B^\beta_s |}\int_0^{t-s}K^s_{u+s}(f)\pare{\frac{\pi}{2}\pare{t-(s+u)}}^{-1/2} dL_u^{0}(|\tilde{W}|)} \\
&=\int_0^\infty \mu^0(dy) {\mathbb E}^{x} \croc{ H_s \int_0^{t-s}K^s_{u+s}(f)\pare{\frac{\pi}{2}\pare{t-(s+u)}}^{-1/2}p\pare{u,|B^\beta_s |,0}du},
\end{split}
\end{equation*}
where we have used Exercise 1.12, Chap. X of \cite{RY} (whose result remains true for symmetric local time).
Performing the change of variable $\xi=\delta\sqrt{t-(s+u)}y$ finally yields
\begin{equation*}
\begin{split}
&{\mathbb E}^x \pare{H_s\, f(B^\beta_t)\id{G^\beta_t > s}}\,=\,\\
& \int_{-\infty}^\infty {\mathbb E}^{x} \croc{ H_s \sum_{\delta\in \{ -1,1\}}\int_0^{t-s}f(\xi)\frac{1+\delta (\beta\circ \sigma_s)(u)}{2}\sqrt{\frac{2}{\pi}}\frac{|\xi |{\rm e}^{-\frac{\xi^2}{2(t-(s+u))}}}{\pare{t-(s+u)}^{3/2}}\frac{{\rm e}^{-\frac{|B^\beta_s |^2}{2u}}}{\sqrt{2\pi u}} du}\,\id{\delta\xi >0}d\xi\\
\end{split}
\end{equation*}

On another hand, for a fixed time $s>0$ we may set $$D^\beta_s:=\inf\{u\geq 0~:~B^\beta_{s+u}=0\}=\inf\{u\geq 0~:~B^\beta_{s} + (W_{s+u} - W_s )=0\}.$$
So that using the equation solved by the inhomogeneous skew Brownian motion (\ref{ISBM}) and the usual Markov property for the standard Brownian motion~:
\begin{equation*}
\begin{split}
&{\mathbb E}^x \pare{H_s f(B^\beta_t)\id{G^\beta_t \leq s}} ={\mathbb E}^x \croc{H_s \,f\pare{B^\beta_s +(W_{s+(t-s)}-W_s)} \id{D^\beta_s \geq (t-s)}}\\
&={\mathbb E}^x \croc{{\mathbb E}^x\croc{H_s \,f\pare{B^\beta_s +(W_{s+(t-s)}-W_s)} \id{D^\beta_s \geq (t-s)}~|~{\cal F}_s}}\\
&={\mathbb E}^x \croc{H_s {\mathbb E}^{B^\beta_s}\croc{\,f\pare{\tilde{W}_{t-s}} \id{\tilde{T}_0 \geq (t-s)}}}\\
&=\int_{-\infty}^\infty dy f(y) {\mathbb E}^x \croc{H_s \frac{1}{\sqrt{2\pi (t-s)}}  \croc{\exp\pare{-\frac{(y-B^\beta_s)^2}{2(t-s)}} - \exp\pare{-\frac{(y+B^\beta_s)^2}{2(t-s)}} }\id{B^\beta_s\,y>0}}.
\end{split}
\end{equation*}
Adding these terms and using the characterization of conditional expectation finally yields that
\begin{equation*}
\begin{split}
{\mathbb E}^x \pare{f(B^\beta_t) | {\cal F}_s}&=\int_{-\infty}^\infty dy f(y) p^{\beta}(s,t ; B^\beta_s, y),
\end{split}
\end{equation*}
where we used the Definition \ref{definition-2}.
\end{proof}

Let us notice that, as $\int_{-\infty}^\infty dy f(y) p^{\beta}(s,t ; B^\beta_s, y)$ is a $\sigma(B^\beta_s)$-measurable random variable, we get from Proposition \ref{prop-markov-process} ,
\begin{equation}
\label{markov-prop-kar}
{\mathbb E}^x \pare{f(B^\beta_t) | {\cal F}_s}={\mathbb E}^x \pare{f(B^\beta_t) | B^\beta_s}=\int_{-\infty}^\infty dy f(y) p^{\beta}(s,t ; B^\beta_s, y).
\end{equation}

This leads naturally to the following important consequence.

\begin{proposition}
\label{prop-markov}
We have

i) The process $B^\beta$ is a Markov process, in the sense of Definition 2.5.10 in \cite{kara}.

ii) For all $x,y\in\R$ we have
\begin{equation}
\label{eq-Psx}
\P^{s,x}(B^\beta_t\in dy)=p^\beta(s,t;x,y)dy.
\end{equation}
\end{proposition}

\begin{remark}
The proposition \ref{prop-markov} will imply in turn that the family $p^\beta(s,t;x,y)$ may be considered as a transition family (t.f.). See the forthcoming Proposition 
\ref{prop:chapman-kolmogorov}.
\end{remark}

Notice that since the considerations of Proposition \ref{prop-markov-process} may be repeated if the fixed time $s$ is replaced by $T$ a $\pare{{\cal F}_t}$-stopping time, we may also state the following~:
\begin{corollary}
The process $B^\beta$ is a strong Markov process in the sense given by Theorem 3.1 in \cite{RY} Chap. III sect.3, p.102.
\end{corollary}

\subsection{Kolmogorov's continuity criterion}
The next result shows a Kolmogorov's continuity criterion for $B^{{{\beta}}}$ uniform w.r.t. the parameter function $\beta(.)$.
\begin{proposition}
\label{prop-kolmogorov}
There exists a universal constant $C>0$ (independent of the function $\beta(.)$) such that for all $\varepsilon\geq 0$ and $t\geq 0$,
\begin{equation}
\label{kolmogorov-2}
{\mathbb E}^{x} |B_{t+\varepsilon}^{{{\beta}}} - B_{t}^{{{\beta}}}|^4 \leq   C\,\varepsilon^2. 
\end{equation}
\end{proposition}
\begin{proof}
Conditioning with respect to $\cF_t$ and using the Markov property and \eqref{eq-Psx} we get,
\begin{equation*}
\begin{split}
{\mathbb E}^{x} |B_{t+\varepsilon}^{{{\beta}}} - B_{t}^{{{\beta}}}|^4&=
\E^x\big[\int_{-\infty}^\infty(y-B^\beta_t)^4\,p^\beta(t,t+\varepsilon;B^\beta_t,y)dy\big]\\
&=\E^x\big[\int_{-\infty}^\infty(y-B^\beta_t)^4\,p(\varepsilon,B^\beta_t,y)dy\\
&\hspace{0.7cm}+\int_{-\infty}^\infty\,dy\,(y-B^\beta_t)^4\;
\int_0^\varepsilon \frac{\beta\circ\sigma_s(u)}{2}\frac{y}{\pi}\frac{e^{-\frac{y^2}{2(\varepsilon-u)}}}{\sqrt{u}(\varepsilon-u)^{3/2}}e^{-\frac{|B^\beta_t|^2}{2u}}du\big]\\
&\leq \E^x\big[\int_{-\infty}^\infty(y-B^\beta_t)^4\,p(\varepsilon,B^\beta_t,y)dy\\
&\hspace{0.7cm}+\int_{-\infty}^\infty\,dy\,(y-B^\beta_t)^4\;
\int_0^\varepsilon \frac{|y|}{2\pi}\frac{e^{-\frac{y^2}{2(\varepsilon-u)}}}{\sqrt{u}(\varepsilon-u)^{3/2}}e^{-\frac{|B^\beta_t|^2}{2u}}du\big]\\
&\leq 2\E^x\big[\int_{-\infty}^\infty(y-B^\beta_t)^4\,p(\varepsilon,B^\beta_t,y)dy\big]\\
\end{split}
\end{equation*}
where we have successively used \eqref{transition-density-complete2} and \eqref{dens-bro2}. As for the brownian density we have

$\int_{-\infty}^\infty(y-B^\beta_t)^4\,p(\varepsilon,B^\beta_t,y)dy\leq C\varepsilon^2$, we get the desired result.

\end{proof}

\section{Existence result for the inhomogeneous skew Brownian motion }
\label{sec-existence}
\subsection{Part I~:~the case of a smooth coefficient $\beta$}
\label{ssec-exist-C1}

\begin{proposition}
\label{exist-C1}
Assume there exist $-1<m<M<1$ s.t. $m\leq \beta(t)\leq M$, for all $t\geq 0$. Assume moreover that $\beta\in\cC^1_b(\R_+)$.
Then there exists a unique strong solution $B^\beta$ to \eqref{ISBM}.
\end{proposition}

Let us introduce some notations. We introduce the $\cC^{1,2}(\R_+\times\R^*)$ function $r(t,y)$ defined by
\begin{equation}
\label{eq-def-r}
r(t,y):=\left\{
\begin{array}{lll}
\dfrac{\beta(t)+1}{2}\,y&\text{if }& y\geq 0\\
\\
\dfrac{1-\beta(t)}{2}\,y&\text{if }& y<0.\\
\end{array}
\right.
\end{equation}

The proof of Proposition \ref{exist-C1} relies on the following lemma.

\begin{lemma}
\label{lem-exist-C1}
Under the assumptions of Proposition \ref{exist-C1} the SDE
\begin{equation}
\label{EDS-rrp}
dY_t=\frac{dW_t}{r'_y(t,Y_t)}-\frac{r'_t(t,Y_t)}{r'_y(t,Y_t)}dt,
\end{equation}
has a unique strong solution.
\end{lemma}

\begin{proof}
The coefficient $(r'_y(t,y))$ is piecewise $\cC^1$ with respect to $y$, measurable with respect to $(t,y)$, uniformly positive, and bounded.
The coefficient $r'_t(t,y)/r'_y(t,y)$ is measurable and bounded. Thus, according to Theorem 1.3 in \cite{legall}, and the remark following it, the SDE
\eqref{EDS-rrp} enjoys pathwise uniqueness and has a weak solution. Therefore the result, using \cite{yamada1,yamada2}. 
\end{proof}

\begin{proof}[Proof of Proposition \ref{exist-C1}]
We set $X_t\defi r(t,Y_t)$ with $Y$ the unique strong solution of \eqref{EDS-rrp} (corresponding to the given Brownian motion $W$). We will show that $X$ solves 
\eqref{ISBM} (with $W$). Therefore the result, as $X$ is $(\cF_t)$ adapted.
Using the change of variable formula proposed by Peskir in \cite{peskir}, we get
$$
\begin{array}{lll}
X_t&=&\displaystyle  r(0,Y_0)+\int_0^t r'_t(s,Y_s)ds+\int_0^tr'_y(s,Y_s)dY_s +\frac{1}{2}\int_0^t\beta(s)dL^0_s(Y)\\
\\
&=&\displaystyle r(0,Y_0)+\int_0^t r'_t(s,Y_s)ds+W_t\\
\\
&&\displaystyle -\int_0^tr'_t(s,Y_s)ds+\frac{1}{2}\int_0^t\beta(s)dL^0_s(Y)\\
\\
&=&\displaystyle r(0,Y_0)+W_t+\frac{1}{2}\int_0^t\beta(s)dL^0_s(Y).\\
\end{array}
$$
It remains to show that $\frac{1}{2}L^0_t(Y)=L^0_t(X)$. To this end we adapt the methodology proposed in Subsection 5.2 of \cite{lejay-sbm}. On one hand the symmetric Tanaka formula gives,
$$
d|X_t|=\mathrm{sgn}(X_t)dX_t+L^0_t(X)=\mathrm{sgn}(Y_t)dW_t+L^0_t(X),$$
where we have used $\mathrm{sgn}(X_t)=\mathrm{sgn}(Y_t)$ and $\mathrm{sgn}(0)=0$ (for the sign function involved in the symmetric Tanaka formula is the symmetric sign function).
On the other hand, using again the formula by Peskir, with the function $(t,y)\mapsto |r(t,y)|$, we get
$$
\begin{array}{lll}
|X_t|=|r(t,Y_t)|&=&\displaystyle  |r(0,Y_0)|+\int_0^t\mathrm{sgn}(Y_s) r'_t(s,Y_s)ds+\int_0^t \mathrm{sgn}(Y_s)dW_s\\
\\
&&\displaystyle -\int_0^t\mathrm{sgn}(Y_s) r'_t(s,Y_s)ds+\frac{1}{2}L^0_t(Y)\\
\\
&=&\displaystyle |r(0,Y_0)|+\int_0^t \mathrm{sgn}(Y_s)dW_s+\frac{1}{2}L^0_t(Y).\\
\end{array}
$$
By unicity of the decomposition of a semi-martingale we get $\frac{1}{2}L^0_t(Y)=L^0_t(X)$, and we are done.

\end{proof}

\subsection{Part II: the general case (proof of Theorem \ref{existence-isbm})}

Till the end of the section we are given a Borel function ${\bf \beta} : {\mathbb R}^+\rightarrow [-1,1]$.

\vspace{0.2cm}

Firstly, combining Propositions \ref{exist-C1} and \ref{prop-markov} we get the following crucial result.

\begin{proposition}
\label{prop:chapman-kolmogorov}
  The family of measures $p^\beta(s,t;x,y)dy$ of Definition \ref{definition-2} is a (inhomogeneous) family of transition probabilities satisfying the Chapman-Kolmogorov equation
\begin{equation}
\label{chapman-kolmogorov}
\int_{-\infty}^\infty p^\beta(s,t;x,y)p^\beta(t,v;y,z)dy= p^\beta(s,v;x,z),\hspace{0,2 cm}0<s<t<v,\,\,x,z\in {\mathbb R}.
\end{equation}
\end{proposition}

\begin{proof}
For a smooth function $\beta(.)$ satisfying the assumptions of Proposition \ref{exist-C1}, there exists a solution to \eqref{ISBM}, which is a Markov process satisfying \eqref{eq-Psx}. Thus, the family $p^\beta(s,t;x,y)dy$ satisfies \eqref{chapman-kolmogorov} (see Chap. III, sect. 1, p. 80 of \cite{RY}).

For the Borel function $\beta(.)$ we may approximate $\beta$ by a sequence of smooth functions $\beta_n(.)$. As the family $p^{\beta_n}(s,t;x,y)dy$ satisfies
\eqref{chapman-kolmogorov} we recover the same result for $p^\beta(s,t;x,y)dy$ thanks to Lebesgue's domination theorem.
\end{proof}

Let us discuss briefly the point that we have reached till now. For any Borel function $\beta$ the family $p^\beta(s,t;x,y)dy$, is a transition family (t.f.). If a strong solution $B^\beta$ to \eqref{ISBM} exists, it is necessarily a Markov process with transition family $p^\beta(s,t;x,y)dy$. 

We will now construct a Markov process with t.f. 
$p^\beta(s,t;x,y)dy$ and show that it provides a weak solution of \eqref{ISBM}.

\vspace{0.2cm}


Thanks to the result of Proposition \ref{prop:chapman-kolmogorov}, we can construct a probability $\tilde{\Q}^0$ on $(\R^{[0,\infty)},{\cal R}^{[0,\infty)})$ such that the coordinate process $(\omega(t))_{t\geq 0}$ is a Markov process with t.f. $p^\beta(s,t;x,y)dy$ (starting from zero) on 
$(\R^{[0,\infty)},{\cal R}^{[0,\infty)},\tilde{\Q}^0)$ (Theorem 1.5 Chap. III
in \cite{RY}).

Using the same computations than in the proof of Proposition \ref{prop-kolmogorov}, it possible to show that a Kolmogorov's continuity criterion holds for 
$(\omega(t))_{t\geq 0}$
under $\tilde{\Q}^0$.
 This implies that we can construct a modification of $(\omega(t))_{t\geq 0}$ with $\tilde{\Q}^0$-a.s. continuous paths. Transporting the measure 
 $\tilde{\Q}^0$ on the set of continuous functions $C$ (see \cite{RY} p.35 for details) we get the following:

\begin{proposition}
\label{prop-process-markov}
 There exists a probability measure $\Q^0$ on $(C,{\cal B}(C))$ under which the coordinate process is a Markov process with t.f. $p^\beta(s,t;x,y)dy$.
            
\end{proposition}

Proposition \ref{prop-process-markov} implies that there exists a Markov process $X^\beta$ with continuous paths and t.f. $p^\beta(s,t,x,y)$,
defined on a probability space $(C,{\cal B}(C),\Q^0)$.

We will denote by $(\cF^\beta_t)_{t\geq 0}$ the complete right continuous filtration endowed by $X^\beta$ on $(C,{\cal B}(C),\Q^0)$.

\vspace{0.2cm}
Having in mind the results of the previous sections, we begin to prove the following proposition~:
\begin{proposition}
We have
$$(|X_t^{ \beta }|)_{t\geq 0}  \stackrel{\cL}{\sim} \pare{|W_t|}_{t\geq 0}.$$ 
\end{proposition}

\begin{proof}
Let $y,x>0$ and $0<s<t$. We compute
\begin{equation*}
\begin{split}
\Q^0(|X^\beta_t|\in dy\,\big|\;|X^\beta_s|=x)&=\Q^0(|X^\beta_t|\in dy\,\big|\;X^\beta_s=x)\Q^0(X^\beta_s>0\big|\;|X^\beta_s|=x)\\
&\hspace{0.2 cm}+\Q^0(|X^\beta_t|\in dy\,\big|\;X^\beta_s=-x)\Q^0(X^\beta_s<0\big|\;|X^\beta_s|=x)\\
&=dy\Big\{\Q^0(X^\beta_s>0\big|\;|X^\beta_s|=x)\big(p^\beta(s,t;x,y)+p^\beta(s,t;x,-y)\big)\\
&\hspace{0.2 cm}+ \Q^0(X^\beta_s<0\big|\;|X^\beta_s|=x)\big(p^\beta(s,t;-x,y)+p^\beta(s,t;-x,-y)\big)\Big\}\\
&=dy\Big\{\Q^0(X^\beta_s>0\big|\;|X^\beta_s|=x)\big(p(t-s,x,y)+p(t-s,x,-y)\big)\\
&\hspace{0.2 cm}+ \Q^0(X^\beta_s<0\big|\;|X^\beta_s|=x)\big(p(t-s,-x,y)+p(t-s,-x,-y)\big)\Big\}\\
&=dy\,[p(t-s,x,y)+p(t-s,x,-y)],
\end{split}
\end{equation*}
where we have used Equation 
\eqref{transition-density-complete2}, and the symmetry of $(x,y)\mapsto p(t,x,y)$ in the computations. The final right hand side expression is the well-known (homogeneous) density of a reflected Brownian motion $|W|$ starting from $x>0$.

As $|X^\beta|$ is Markov as well as $|W|$, and since both processes have continuous sample paths, we get the desired result (see for example Theorem 1.5 Chap. III
in \cite{RY}). 
\end{proof}
Consequently, we have the following theorem~:~
\begin{theorem}
\label{solution-faible}
 There exists a weak solution to \eqref{ISBM}.
\end{theorem}

\begin{proof}
Let $X^\beta$ be the Markov process discussed after Proposition  \ref{prop-process-markov}. Let $s<t$.
The Markov property yields (expectations are computed under $\Q^0$)~: 
\begin{equation*}
\begin{split}
{\mathbb E}^0 \pare{X^{\beta}_t | {\cal F}^\beta_s} &= \int_{-\infty}^\infty y\, p^{\beta}(s,t ; X^{\beta}_s, y) dy\\
&=\int_0^{t-s}\int_{-\infty}^\infty\frac{y|y|}{\sqrt{2\pi}}\frac{e^{-\frac{y^2}{2(t-s-u))}}}{(t-s-u)^{3/2}}dy\frac{e^{-|X^\beta_s|^2}}{\sqrt{2\pi u}}du\\
&\hspace{0.3cm}+\int_0^{t-s}\frac{\beta\circ\sigma_s(u)}{t-s-u}\pare{\int_{-\infty}^\infty|y|^2\frac{e^{-\frac{y^2}{2(t-s-u))}}}{\sqrt{2\pi(t-s-u)}}dy}\frac{e^{-|X^\beta_s|^2}}{\sqrt{2\pi u}}du\\
&\hspace{0.3cm}+\int_{-\infty}^\infty\frac{y}{\sqrt{2\pi(t-s)}}e^{-\frac{(y-X^\beta_s)^2}{2(t-s)}}dy\\
&=X^{\beta}_s + \int_0^{t-s}\beta\circ \sigma_s(u)\frac{{\rm e}^{-\frac{|X^{\beta}_s |^2}{2u}}}{\sqrt{2\pi u}} du. \\
\end{split}
\end{equation*}
Note that $X^{\beta}$ is a Markov process and $|X^{\beta}|$ is a reflected Brownian motion. Thus, $X^{\beta}$ admits a symmetric local time, which is a continuous additive functional of $X^{\beta}$. So that for $s<t$~:
$$
{\mathbb E}^0 \pare{\int_0^t \beta(u)dL^0_u\pare{X^{\beta}} | {\cal F}^{\beta}_s} = \int_0^s \beta(u)dL^0_u\pare{X^{\beta}} + {\mathbb E}^0 \pare{\int_s^{t} \beta (u)dL^0_u\pare{X^{\beta }}  | {\cal F}^{\beta}_s }.
$$
But,
\begin{equation*}
\begin{split}
 \E^0 \pare{\int_s^t \beta(u)dL^0_u\pare{X^{\beta}}   | {\cal F}^{\beta}_s  } 
&={\mathbb E}^{0} \pare{\int_0^{t-s} \beta\circ \sigma_s (u)dL^0_u\pare{X^{\beta\circ\sigma_s}}\,|\,X^{\beta}_s}\\
&={\mathbb E}^{0} \pare{\int_0^{t-s} \beta\circ \sigma_s (u)dL^0_u\pare{|\tilde{W}|}\,|\,X^{\beta}_s}\\
&=\int_0^{t-s}\beta\circ \sigma_s(u)\frac{{\rm e}^{-\frac{|X^{\beta}_s |^2}{2u}}}{\sqrt{2\pi u}} du.
\end{split}
\end{equation*}
Combining these facts ensures that $\left \{X_t^{\beta} - \int_0^t \beta(u)dL^0_u\pare{X^{\beta}}~:~t\geq 0\right \}$ is a $({\cal F}^\beta_{t})$ local martingale. Since $\langle X^{\beta} \rangle_t =\langle |X^{\beta}| \rangle_t = t$, we deduce that $\left \{X_t^{\beta} - \int_0^t \beta(u)dL^0_u\pare{X^{\beta}}~:~t\geq 0\right \}$
is in fact a $({\cal F}^\beta_t)$ Brownian motion (under $\Q^0$), ensuring that $X^{\beta}$ satisfies (\ref{ISBM}).
\end{proof}

The existence of a weak solution, together with the pathwise uniqueness stated in the theorem \ref{theo-fund-Weinryb}, ensures the existence of a unique strong solution to \eqref{ISBM} (see \cite{yamada1,yamada2}). It is clear that this solution is a strong Markov process with t.f. $p^\beta(s,t;x,y)dy$ (Section \ref{sec-markov}). Therefore Theorem \ref{existence-isbm} is proved.

\vspace{0.3cm}

The properties of the t.f. $p^\beta(s,t;x,y)$ allow us to state the following proposition:

\begin{proposition}
\label{prop-borel}
Let $\tilde{\beta} : {\mathbb R}^+\rightarrow [-1,1]$ a Borel function such that 
$$
\tilde{\beta}(t)=\beta(t)\hspace{0.2 cm}\text{for Lebesgue almost every }t\in [0,1].\quad(*)
$$
Then, 
$B^{\tilde{\beta}}$ and $B^{\beta}$ are equivalent.
\end{proposition}
\begin{proof}
Since $B^{\tilde{\beta}}$ and $B^{\beta}$ are Markovian processes with the same probability transition function $p^\beta(s,t;x,y)$ given by definition \ref{definition-2}, they are equivalent. 
\end{proof}

\section{Proof of Theorem \ref{theo-loi}}
The result of Theorem \ref{theo-loi} will appear as a consequence of the previous Proposition \ref{prop-loi-jointe-signe-g1-m1} and the following lemma which appears again as a consequence of Proposition \ref{prop-fund-Weinryb} and known results concerning the standard Brownian motion~:~
\begin{lemma}
\label{lemme-temps-local-bridge}
Under ${\mathbb P}^0$, the process $\{|\check{B}_t^\beta|:=\frac{1}{\sqrt{G^\beta_1}}|B_{t\,G^\beta_1}^\beta|~:~t\leq 1\}$ is the reflection (above $0$) of a Brownian Bridge independent of $\,{\cal G}:=\sigma\left \{G_1^\beta, B^\beta_{G_1^\beta+u}~;~u\geq 0\right \}$.
\end{lemma}
\begin{proof}
By the result of Theorem \ref{prop-fund-Weinryb} and time inversion, $|B^\beta_t| := t\,|\tilde{B}^\beta_{\frac{1}{t}}|$ is a reflected Brownian motion. Note that 
$$\tilde{d}^\beta_1:=\inf\pare{u>1~:~|\tilde{B}^\beta_u|=0} = \frac{1}{G_1^\beta}.$$ So that, since $\tilde{B}^\beta_{\frac{1}{G^\beta_1}} = \tilde{B}^\beta_{\tilde{d}^\beta_1}=0 $,
\begin{equation*}
\begin{split}
\frac{1}{\sqrt{G^\beta_1}}|B_{t\,G^\beta_1}^\beta| &= t\sqrt{G^\beta_1}|\tilde{B}^\beta_{\frac{1}{t\,G^\beta_1}}| = \frac{t}{\sqrt{\tilde{d}^\beta_1}}|\tilde{B}^\beta_{\frac{1}{G^\beta_1} + \frac{1}{G^\beta_1}\pare{\frac{1}{t}-1}} - \tilde{B}^\beta_{\frac{1}{G^\beta_1}}|\\
&= \frac{t}{\sqrt{\tilde{d}^\beta_1}}|\tilde{B}^\beta_{\tilde{d}^\beta_1+\tilde{d}^\beta_1\pare{\frac{1}{t}-1}} - \tilde{B}^\beta_{\tilde{d}^\beta_1}|.
\end{split}
\end{equation*}
Since $\left \{|\tilde{B}^\beta_{\tilde{d}^\beta_1+u} - \tilde{B}^\beta_{\tilde{d}^\beta_1}|~:~u\geq 0\right \}$ is a reflected Brownian motion independent of 
$\tilde{{\cal F}}_{\tilde{d}^\beta_1}$ and $\tilde{B}^\beta_{\tilde{d}^\beta_1}=0$, the process $|\hat{B}^\beta_u|:=\frac{1}{\sqrt{\tilde{d}^\beta_1}}|\tilde{B}^\beta_{\tilde{d}^\beta_1+\tilde{d}^\beta_1 u}|$ is also a reflected Brownian motion independent of $\tilde{{\cal F}}_{\tilde{d}^\beta_1}$~;~hence, 
$\pare{t|\hat{B}^\beta_{\frac{1}{t}-1}|}_{t\geq 0}$ is a reflected Brownian Bridge independent of $\tilde{{\cal F}}_{\tilde{d}^\beta_1}$. This implies the result. 
\end{proof}
\begin{corollary}
We have that under ${\mathbb P}^0$,
\begin{equation}
L_1^0\pare{B^\beta} = \sqrt{G_1^\beta}\ell^0_1
\end{equation}
where $\ell^0_1$ is the symmetric local time at time $1$ of a standard Brownian Bridge independent of ${\cal G}$.
\end{corollary}
\begin{proof}
Under ${\mathbb P}^0$, we have that
\begin{equation*}
\begin{split}
&L_1^0\pare{B^\beta} = L_1^0\pare{|B^\beta |}=L_{G_1^\beta}^0\pare{|B^\beta |}\\
&=L_{G_1^\beta}^0\pare{\sqrt{G_1^\beta}|\check{B}^\beta_{./G_1^\beta}|}.
\end{split}
\end{equation*}
Recall that the symmetric local time a semimartingale $(Y_t)$ is given by
 $$L^0_t(Y) = \lim_{\varepsilon\rightarrow 0}\frac{1}{2\varepsilon}\int_0^t\id{(-\varepsilon, \varepsilon)}(Y_s)d\langle Y\rangle _s.$$
  So we find, using an obvious change of variable, that
\begin{equation*}
\begin{split}
&L_1^0\pare{B^\beta} =
\sqrt{G_1^\beta}L_{1}^0\pare{|\check{B}^\beta |}\,= \sqrt{G_1^\beta}\ell^0_1,
\end{split}
\end{equation*}
where the last equality comes from the result of Lemma \ref{lemme-temps-local-bridge}.
\end{proof}
Combining this result and the result of Proposition \ref{prop-loi-jointe-signe-g1-m1} gives that
$$
\pare{G_1^\beta, L^0_1\pare{B^\beta}, B^\beta_1} \stackrel{{\cal L}}{\sim} \pare{G_1, \sqrt{G_1}\ell^0_1, Y\sqrt{1-G_1}M_1}
$$ 
where $G_1\stackrel{{\cal L}}{\sim} {\rm Arcsin}$, $\ell^0_1\stackrel{{\cal L}}{\sim}\sqrt{2{\rm \bf e}}$, $M_1\stackrel{{\cal L}}{\sim}\sqrt{2{\rm \bf e}}$ are independent and where $Y$ denotes a r.v. independent of $M_1$ and $\ell^0_1$ and satisfying
$$
\cL\pare{Y\,|\,G_1=s}\stackrel{\cL}{\sim} \cR\pare{\frac{1+{\bf \beta}(s)}{2}}.
$$
This construction gives the result announced in Theorem \ref{theo-loi}.

\begin{remark}
We proved the result of Theorem \ref{theo-loi} using only probabilistic tools and not the computations of Section \ref{sec:fourier}. Since we used only that $B^\beta$ is solution of (\ref{ISBM}), the arguments developed in this section yield another proof of Proposition \ref{prop-densityzero} as a by product.
\end{remark}


\section{Another construction of the inhomogeneous skew Brownian motion}

\subsection{Construction of the Inhomogeneous SBM with piecewise constant coefficient $\beta$ from a reflected Brownian motion}

Let $\{\Pi~:~0=t_0<t_1<\dots<t_i<\dots <t_n=1\}$ be a partition of the interval $[0,1]$.

We define a function $i:[0,1]\to\llbracket 0:n-1\rrbracket$ by
$$i(t)=\sup\{ 0\leq k\leq n-1:\,t_k\leq t  \},\quad\forall t\in[0,1)\quad\text{and}\quad i(1)=n-1.$$

Let $\bar{\beta}:\R^+\to [0,1]$ be a r.c.l.l. function with constant value in $[-1,1]$ on each interval $[t_i,t_{i+1})$. In particular $\bar{\beta}$ is a Borel function. In this section, we give a construction of a weak solution of (\ref{ISBM}) on the interval $[0,1]$, obtained by changing the sign of the excursion of a 
reflecting Brownian motion. We are inspired by \cite{RY}, Chap. XII, Exercise 2.16 p. 487. 

Let us follow the notations of \cite{RY} concerning the excursions of a Brownian motion $B$: the excursion process is denoted by
$(e_s)_{s\geq 0}$, where the index $s$ is in the local time scale. Each excursion $e_s(\omega)$ has support $[\tau_{s-}(\omega),\tau_s(\omega))$, 
where $\tau_s(\omega)=\sum_{u\leq s}R(e_u(\omega))$, and $\tau_{s-}(\omega)=\sum_{u< s}R(e_u(\omega))$,
with $R(e_s(\omega))$ the length of the excursion $e_s(\omega)$. We recall that $L^0_t(B)$ can be recovered as the inverse of $\tau_t$.

\vspace{0.2cm}

The construction is the following~:~for each $0\leq i\leq n-1$ let $(Y^i_k)_k$ be a sequence of independent r.v.'s, identically distributed with law 
$\cR\pare{\frac{1+ \beta^n_i}{2}}$ (with $\beta^n_i\defi\bar{\beta}(t_i)$), and defined on some probability space $(\Omega,\cF,\P)$. Let $B$ be a standard Brownian motion independent of the $Y^i_k$'s, constructed on $(\Omega,\cF,\P)$.
The set of its excursions $e_s(\omega)$ is countable and may be given the ordering of $\N$.

We define a process $X^{\bar{\beta}}$ on $[0,1]$ by putting
$$
\forall t\in [0,1],\quad X^{\bar{\beta}}_t(\omega)=Y^{i(\tau_{s-}(\omega))}_{k_s(\omega)}(\omega)|e_s(t-\tau_{s-}(\omega),\omega)|,
$$
if $\tau_{s-}(\omega)\leq t\leq \tau_{s}(\omega)$ and $e_s(\omega)$ is the $k_s(\omega)$-th excursion in the above ordering.


For $\tau_{s-}(\omega)\leq t\leq \tau_{s}(\omega)$ we have $\tau_{s-}(\omega)=g_t(\omega)$ where $g_t\defi\sup\{u<t~:~|B_u| = 0\}$, and
$|e_s(t-\tau_{s-}(\omega),\omega)|=|B_t(\omega)|$. 

Note also that for a fixed $\omega \in \Omega$, the construction does not make a use of the entire double indexed sequence $(Y^i_k(\omega))_{i\in \{0,\dots,n-1\}, k\in \mathbb N}$.

\begin{proposition}
\label{prop:identification}
The process $X^{\bar{\beta}}$ is a weak solution of equation (\ref{ISBM}) with parameter $\bar{\beta}$ and starting from zero.
\end{proposition}

\begin{proof}
We only detail the main arguments of the proof and leave some of the details to the reader.

{\it 1st step : preliminary facts}

Note that $X^{\bar{\beta}}$ is constructed such that $$(|X_t^{\bar{\beta} }|)_{t\geq 0}  \stackrel{\cL}{\sim} \pare{|W_t|}_{t\geq 0}.$$
Moreover, defining $ G_1^{\bar{\beta}}:=\sup\pare{0\leq s\leq 1~:~X_s^{\bar{\beta}} = 0}$ and $M_1^{\bar{\beta}}:=|X_1^{\bar{\beta}}|/\sqrt{1-G_1^{\bar{\beta}}}$, we see from the construction of $X^{\bar{\beta}}$ that
\begin{equation}
\pare{{\rm sgn}(X^{\bar{\beta}}_1), G_1^{\bar{\beta}}, M^{\bar{\beta}}_1}\stackrel{\cL}{\sim}\pare{Y, G_1, M_1},
\end{equation}
where $G_1\stackrel{\cL}{\sim}\text{\rm Arcsin}$, $M_1\stackrel{\cL}{\sim}\sqrt{2{\rm \bf e}}$,
$G_1$ and $M_1$ are independent,
and where $Y$ denotes some r.v. independent of $M_1$ satisfying
$$
\cL\pare{Y\,|\,G_1=s}\stackrel{\cL}{\sim} \cR\pare{\frac{1+{\bar{\beta}}(s)}{2}}.
$$
Let $\pare{\bar{\cal F}_t}$ the natural filtration of $X^{\bar{\beta}}$ that has been completed and augmented in order to satisfy the usual conditions. From the construction of $X^{\bar{\beta}}$, we see that
\begin{equation}
\label{projection-signe-2}
{\mathbb E}\pare{{\rm sgn}(X^{\bar{\beta}}_t)~|~\bar{\cal F}_{G_t^{\bar{\beta}}}} = {\mathbb E}\pare{ Y^{i(G^{\bar{\beta}}_t)}_{k_{L^{0}_{t}(|B|)}}~|~\bar{\cal F}_{G_t^{\bar{\beta}}}} = \bar{\beta}(G_t^{\bar{\beta}}).
\end{equation}
\hfill
\vspace{0,2 cm}

{\it 2nd step : $X^{\bar{\beta}}$ is a $\pare{\bar{\cal F}_t}$ Markov process}

Because of these preliminary facts, we may repeat the arguments of the first part in the proof of Proposition \ref{prop-markov-process}~:~we see that for $s<t$ and any measurable function $f$~:~
\begin{equation*}
\begin{split}
&{\mathbb E} \pare{f(X^{\bar{\beta}}_t)\id{G^\beta_t > s} | \bar{\cal F}_{s}}\\
&=\int_{-\infty}^\infty d\xi f(\xi){\mathbb E}^{0} \croc{ \sum_{\delta\in \{ -1,1\}}\int_0^{t-s}\frac{1+\delta (\beta\circ \sigma_s)(u)}{2}\sqrt{\frac{2}{\pi}}\frac{|\xi |{\rm e}^{-\frac{\xi^2}{2(t-(s+u))}}}{\pare{t-(s+u)}^{3/2}}\frac{{\rm e}^{-\frac{|X^{\bar{\beta}}_s |^2}{2u}}}{\sqrt{2\pi u}} du}\,\id{\delta\xi>0 }.
\end{split}
\end{equation*}

The part ${\mathbb E} \big(f(X^{\bar{\beta}}_t)\id{G^\beta_t \leq s} | \bar{\cal F}_{s}\big)$ is more complicated since we cannot refer to  Equation (\ref{ISBM}).

Still, for a fixed time $s>0$ we may set $$D^{\bar{\beta}}_s:=\inf\{u\geq 0~:~X^{\bar{\beta}}_{s+u}=0\}.$$
We have~:
\begin{equation*}
\begin{split}
D^{\bar{\beta}}_s&=\inf\{u\geq 0~:~X^{\bar{\beta}}_{s+u}=0\}\\
&=\inf\{u\geq 0~:~X^{\bar{\beta}}_{s} + X^{\bar{\beta}}_{s+u} - X^{\bar{\beta}}_{s} =0\}\\
&=\,\inf\{u\geq 0~:~X^{\bar{\beta}}_{s} + {\rm sgn}(X^{\bar{\beta}}_{s})\pare{|B_{s+u}| - |B_{s}|}=0\}\\
&=\,\inf\{u\geq 0~:~\pare{|B_{s+u}| - |B_{s}|}=- |X^{\bar{\beta}}_{s}|\}.
\end{split}
\end{equation*}
But on the set $\{G^\beta_t \leq s\} = \{D^{\bar{\beta}}_s \geq (t-s)\}$ and for $r<D^{\bar{\beta}}_s$, the random variables $B_{s+r}$ and $B_{s}$ share the same sign. We deduce that on  the set $\{G^\beta_t \leq s\}$,
\begin{equation*}
\begin{split}
D^{\bar{\beta}}_s&
=\left \{
\begin{array}{l}
\inf\{u\geq 0~:~\pare{B_{s+u} - B_{s}}=-|X^{\bar{\beta}}_{s}|\}:= T_s^{+}~\hspace{0,4 cm}\text{if }B_s\geq 0~;\\
\inf\{u\geq 0~:~\pare{B_{s+u} - B_{s}}= |X^{\bar{\beta}}_{s}|\}:= T_s^-~\hspace{0,4 cm}\text{if }B_s < 0.
\end{array}
\right .
\end{split}
\end{equation*}

Let us introduce ${\cal K}_s := \bar{\cal F}_s\vee \sigma\pare{B_s}$. 

We have
\begin{equation*}
\begin{split}
&{\mathbb E} \pare{f(X_t^{\bar{\beta}})\id{G^\beta_t \leq s}| {\cal K}_{s}} ={\mathbb E}\pare{f\pare{X^{\bar{\beta}}_s +{\rm sgn}(X^{\bar{\beta}}_{s})\pare{|B_{s+(t-s)}| - |B_{s}|}} \id{D^{\bar{\beta}}_s \geq (t-s)}| {\cal K}_{s}}\\
&=\id{B_s\geq 0}{\mathbb E}\pare{f\pare{X^{\bar{\beta}}_s +{\rm sgn}(X^{\bar{\beta}}_{s})\pare{B_{s+(t-s)} - B_{s}}} \id{T_s^+ \geq (t-s)}| {\cal K}_{s}}\\
&\hspace{0,3 cm}+\id{B_s< 0}{\mathbb E}\pare{f\pare{X^{\bar{\beta}}_s -{\rm sgn}(X^{\bar{\beta}}_{s})\pare{B_{s+(t-s)} - B_{s}}} \id{T_s^- \geq (t-s)}| {\cal K}_{s}}.
\end{split}
\end{equation*}

Since $\pare{B_{s+u} - B_{s}~:~u\geq 0}$ is a Brownian motion independent of $B_s$ and of $\bar{\cal F}_{s}$ (and thus of ${\cal K}_s$), we may integrate this expression using the known laws of the Brownian motion killed when hitting $0$~:
\begin{equation*}
\begin{split}
&\id{B_s\geq 0}{\mathbb E}\pare{f\pare{X^{\bar{\beta}}_s +{\rm sgn}(X^{\bar{\beta}}_{s})\pare{B_{s+(t-s)} - B_{s}}} \id{T_s^+ \geq (t-s)}| {\cal K}_{s}}\\
&=\id{B_s\geq 0}\int_{-\infty}^\infty d\theta f\pare{X^{\bar{\beta}}_s +{\rm sgn}(X^{\bar{\beta}}_{s})\pare{\theta-|X_s^{\bar{\beta}}|}}\\
&\hspace{0,3 cm}\times \frac{1}{\sqrt{2\pi (t-s)}}  \croc{\exp\pare{-\frac{(\theta-|X^{\bar{\beta}}_s|)^2}{2(t-s)}} - \exp\pare{-\frac{(\theta+|X^{\bar{\beta}}_s|)^2}{2(t-s)}} }\id{|X^{\bar{\beta}}_s|\,\theta>0}\\
&=\id{B_s\geq 0}\int_{-\infty}^\infty dy f\pare{y}\frac{1}{\sqrt{2\pi (t-s)}}  \croc{\exp\pare{-\frac{(y-X^{\bar{\beta}}_s)^2}{2(t-s)}} - \exp\pare{-\frac{(y+X^{\bar{\beta}}_s)^2}{2(t-s)}} }\id{X^{\bar{\beta}}_s y>0}
\end{split}
\end{equation*}
where for the last line, we performed the change of variable $y = {\rm sgn}(X_s^{\bar{\beta}}) \theta$.
We have a similar term on the side $\{B_s < 0\}$, so that 
\begin{equation*}
\begin{split}
&{\mathbb E} \pare{f(X_t^{\bar{\beta}})\id{G^\beta_t \leq s}| {\cal K}_{s}}\\
& = \int_{-\infty}^\infty dy f\pare{y}\frac{1}{\sqrt{2\pi (t-s)}}  \croc{\exp\pare{-\frac{(y-X^{\bar{\beta}}_s)^2}{2(t-s)}} - \exp\pare{-\frac{(y+X^{\bar{\beta}}_s)^2}{2(t-s)}} }\id{X^{\bar{\beta}}_s y>0}\\
&={\mathbb E} \pare{f(X_t^{\bar{\beta}})\id{G^\beta_t \leq s}| \bar{\cal F}_{s}}.
\end{split}
\end{equation*}
Finally, adding both parts gives that
\begin{equation*}
\begin{split}
{\mathbb E} \pare{f(X^{\bar{\beta}}_t) | \bar{{\cal F}}_s}&=\int_{-\infty}^\infty dy f(y) p^{\bar{\beta}}(s,t ; X^{\bar{\beta}}_s, y).
\end{split}
\end{equation*}
This is enough to conclude that $X^{\bar{\beta}}$ is a Markov process with $p^{\bar{\beta}}(s,t;x,y)dy$ as its family of transition probability (satisfying (\ref{chapman-kolmogorov})).

\hfill
\vspace{0,2 cm}

{\it 3rd step : the process $X^{\bar{\beta}}$ is a weak solution of equation (\ref{ISBM})}

It suffices to perform the same computations as in the proof of Theorem \ref{solution-faible}. Indeed $X^{\bar{\beta}}$ is Markov and, by construction, $|X^{\bar{\beta}}|$ is a reflected Brownian motion.

\end{proof}


   
\subsection{The convergence result}
\label{sec:importante}

In this section, we show another way of constructing a weak solution of (\ref{ISBM}) on the time interval $[0,1]$.

\subsubsection{The technical assumption $\cH$}

In this subsection we state the assumptions of Theorem \ref{convergence}.

Let $\{\Pi_n~:~0=t^n_0<t^n_1<\dots<t^n_i<\dots t^n_n=1,\hspace{0.2 cm}n\geq 0\}$ a sequence of partitions over $[0,1]$. Assume that
$$
\sup_{0\leq i\leq n-1}|t^n_{i+1} -t^n_{i}|\xrightarrow[n\rightarrow+\infty]{} 0.
$$
We now state a technical assumption $\cH$ that will be used in force in all our considerations.

It is possible to construct a decreasing (resp. increasing) sequence 
$({\bf \check{\beta}}_n)$  (resp. $({\bf \hat{\beta}}_n)$) of r.c.l.l. step functions that are constant on each of the intervals $[t_i^n, t^n_{i+1})$ in such a way that
$$
{\bf \check{\beta}}_n(t) \geq {\bf \beta}(t)\geq {\bf \hat{\beta}}_n(t)\hspace{0,4 cm}
\text{and} \hspace{0,4 cm}
\lim_{n\rightarrow +\infty}{{\bf \check{\beta}}_n(t)} = \lim_{n\rightarrow +\infty}{{\bf \hat{\beta}}_n(t)} = {\bf {\beta}}(t),\hspace{0.2 cm}\forall t\in [0,1].
$$

\subsubsection{A convergence result}

Corresponding to such sequences $({\bf \check{\beta}}_n)$ and $({\bf \hat{\beta}}_n)$, we introduce the corresponding sequences $(B^{{\bf \check{\beta}}_n})_{n\geq 0}$ and $(B^{ \hat{\beta}_n})_{n\geq 0}$ of ISBM which are strong solutions to equation (\ref{ISBM}) and driven by the \underline{same} Brownian motion $W$. 

Then, we have the following comparison principle~:~
\begin{proposition}
\label{prop-increasing}
For any $n\geq 0$,
\begin{equation*}
{\P}\left [B_t^{{\bf \check{\beta}}_n}\geq B_t^{{\bf \check{\beta}}_{n+1}},\hspace{0.2 cm}\forall t\in [0,1]\right ] = 1,
\end{equation*}
and
\begin{equation*}
{\P}\left [B_t^{{\bf \hat{\beta}}_n}\leq B_t^{{\bf \hat{\beta}}_{n+1}},\hspace{0.2 cm}\forall t\in [0,1]\right ] = 1.
\end{equation*}
\end{proposition}
\begin{proof}
 Since $\check{\bf \beta}_n$ (resp. $\hat{\bf \beta}_n$) is a step function, the process $B^{\check{\bf \beta}_n}$ can be viewed as a concatenation of (homogeneous) skew Brownian motions on each interval $[t^n_i,t^n_{i+1}[$. Thus the result is a direct consequence of the comparison principle for the SBM (see \cite{legall}, p.73).
\end{proof}

\begin{remark}
In the above proof it is convenient to see the existing process $B^{\check{\bf \beta}_n}$ as a concatenation of SBM's. Conversely we may wish to take the concatenation of SBM's as the starting point of the construction of an ISBM with piecewise constant coefficient. However we did not manage to exploit this idea, because the flow of a classical skew Brownian motion is not defined for all starting points $x$ simultaneously (see the remark in the introduction of \cite{B-C-flow}, previously cited).
\end{remark}

Let us now state the main result of this section,
\begin{theorem}
\label{convergence}
Assume that $\cH$ is satisfied.

Then,
\begin{equation*}
{\P}\left [\lim_{n\rightarrow +\infty}\sup_{t\in [0,1]}|B_t^{{\bf \check{\beta}}_n} -B_t^{\beta}| = \lim_{n\rightarrow +\infty}\sup_{t\in [0,1]}|B_t^{{\bf \hat{\beta}}_n} - B_t^{\beta}| = 0 \right ] = 1.
\end{equation*}
\end{theorem}
\begin{proof}
\medskip
{\it 1st step : convergence UCP}
\medskip

Let $t\in [0,1]$ fixed. The result of Proposition \ref{prop-increasing} implies that
$(B_t^{{\bf \check{\beta}}_n})_{n\geq 0}$ (resp. $(B_t^{{\bf \hat{\beta}}_n})_{n\geq 0}$) is an a.s. decreasing (resp. increasing) sequence of random variables bounded from below (resp. bounded from above). Thus, the sequence $(B_t^{{\bf \check{\beta}}_n})_{n\geq 0}$ (resp. $(B_t^{{\bf \hat{\beta}}_n})_{n\geq 0}$) converges a.s. to some random variable
$\check{Y}_t$ (resp. $\hat{Y}_t$). For simplicity, we concentrate in the rest of the proof to the family of random variables $\{\check{Y}_t~:~0\leq t\leq 1\}$.

From Lebesgue's theorem and the result of Proposition \ref{prop-kolmogorov}, we have for $t\in [0,1]$~:
\begin{equation}
\label{kolmogorov-2}
\begin{split}
{\mathbb E} |\check{Y}_{t+\varepsilon} - \check{Y}_t|^4 &= {\mathbb E} \lim\limits_{n\rightarrow +\infty}|B^{\check{\bf \beta}_n}_{t+\varepsilon} - B_t^{\check{\bf \beta}_n}|^4\leq C \varepsilon^2.
\end{split}
\end{equation}
Kolmogorov's continuity criterion implies that there is a continuous modification of the process $\check{Y}$ (that we still denote abusively $\check{Y}$). Moreover, Proposition \ref{prop-increasing} implies that the sequence $(B_{.}^{{\bf \check{\beta}}_n})_{n\geq 0}$ is an a.s decreasing sequence of continuous functions on the compact space $[0,1]$ converging simply to a continuous function $\check{Y}_.$~:~from Dini's theorem this convergence is uniform almost surely. 
Consequently, $(B_{.}^{{\bf \check{\beta}}_n})_{n\geq 0}$ is a.s. a Cauchy sequence for the uniform norm over $[0,1]$. Moreover, Lebesgue's dominated convergence theorem ensures that
\begin{equation}
\label{uniform-L1}
{\mathbb E}\left [\sup_{t\in [0,1]}|B_{t}^{{\bf \check{\beta}}_p}-B_{t}^{{\bf \check{\beta}}_q}|\right ]\xrightarrow[p,q\rightarrow +\infty]{} 0.
\end{equation}
From (\ref{uniform-L1}), we see that $(B_{.}^{{\bf \check{\beta}}_n})_{n\geq 0}$ is a Cauchy sequence in the complete space ${\mathbb D}_{\rm ucp}$ (see for example \cite{Protter} Chap. II, p. 57). Combining these facts ensures that the family $(\check{Y}_t)$ agregates as a process with a.s. continuous trajectories and we have proved that~:~
\begin{equation}
\label{uniform-L1-2}
\sup_{t\in [0,1]}|B_t^{{\bf \check{\beta}}_n} -\check{Y}_t| \xrightarrow[n\rightarrow +\infty]{\P- a.s.} 0.
\end{equation}

\medskip
{\it 2nd step : identification of the limit}
\medskip

Let us proceed to identify $\pare{\check{Y}_t}_{t\in [0,1]}$ with the unique strong solution of (\ref{ISBM}) (relative to the Brownian motion
$W$).

On one hand, from the fact that $(|B_t^{{\bf \check{\beta}}_n }|)_{t\geq 0}  \stackrel{\cL}{\sim} \pare{|W_t|}_{t\geq 0}$ and from (\ref{uniform-L1-2}), we deduce that
$(|\check{Y}_t|)_{t\in [0,1]}\stackrel{\cL}{\sim} \pare{|W_t|}_{t\in [0,1]}$. In particular, $\pare{|\check{Y}_t|}_{t\in [0,1]}$ is a semimartingale and admits a symmetric local time process $L_.^{0}(|\check{Y}|)$.

A consequence of the (symmetric) Tanaka formula is that
\begin{equation}
\label{ito-tanaka}
\begin{split}
|B_t^{{\bf \check{\beta}}_n}| &= |x| + \int_0^t {\rm sgn}(B_s^{{\bf \check{\beta}}_n})dW_s + L^0_t(B^{{\bf \check{\beta}}_n})\hspace{0,4 cm}(\text{with }\;\;{\rm sgn}(0)=0).\\
\end{split}
\end{equation}
As $|\check{Y}|$ is a reflected brownian motion we have
\begin{equation}
\label{ito-tanakaY-1}
\begin{split}
|\check{Y}_t| &= |x| + \tilde{W}_t + L_t^{0}\pare{|\check{Y}|},
\end{split}
\end{equation}
for some brownian motion $\tilde{W}$. But
 from the a.s. uniform convergence of $(B_t^{{\bf \check{\beta}}_n})_{t\in [0,1]}$ towards $(\check{Y}_t)_{t\in [0,1]}$ and the dominated convergence theorem for stochastic integrals (see for example Theorem 2.12 p.142 in \cite{RY}), 
 we can see that there is a finite variation process $A$ s.t. 
 $$\sup_{s\in[0,1]}|L^0_s(B^{{\bf \check{\beta}}_n})-A_s|\xrightarrow[n\rightarrow +\infty]{\P}0, $$
 and that
 $$
 \sup_{s\in[0,1]}\big|\,|B_s^{{\bf \check{\beta}}_n}|-(|x| + \int_0^s {\rm sgn}(\check{Y}_u)dW_u+A_s)\,\big|\xrightarrow[n\rightarrow +\infty]{\P}0.
 $$
 Thus,
 \begin{equation}
 \label{ito-tanakaY-2}
 |\check{Y}_t|=|x| + \int_0^t {\rm sgn}(\check{Y}_u)dW_u+A_t.
 \end{equation}
 Using \eqref{ito-tanakaY-1} and \eqref{ito-tanakaY-2}, and the unicity of the Doob decomposition of a semimartingale yields that
 \begin{equation}
 \label{ito-tanakaY-3}
 |\check{Y}_t|=|x| + \int_0^t {\rm sgn}(\check{Y}_s)dW_s+L_t^{0}\pare{|\check{Y}|}.
 \end{equation}

 Note that we have proven that,
\begin{equation*} 
\left [\sup_{t\in [0,1]}\left |L^0_t(B^{{\bf \check{\beta}}_n}) -L^{0}_t(|\check{Y}|) \right |>\varepsilon\right ] \xrightarrow[n\rightarrow +\infty]{\P}0.
\end{equation*}

On another hand, proceeding as in Section \ref{sec-existence}, it is not difficult to see from the a.s. uniform convergence of $(B_t^{{\bf \check{\beta}}_n})_{t\in [0,1]}$ towards $(\check{Y}_t)_{t\in [0,1]}$ that $(\check{Y}_t)_{t\in [0,1]}$ is a Markov process with $p^\beta(s,t;x,y)dy$ as its t.f. Hence, we may proceed just as in Step 3 of Proposition \ref{prop:identification} and we see that there exists a Brownian motion $\check{W}$ such that 
\begin{equation}
\label{ito-tanaka-2}
\check{Y}_t = x + \check{W}_t + \int_0^t \beta(s) dL^{0}_s(|\check{Y}|) = x + \check{W}_t + \int_0^t \beta(s) dL^{0}_s(\check{Y}).
\end{equation}
It remains to identify $\check{W}$ with $W$.

From \eqref{ito-tanakaY-3} and \eqref{ito-tanaka-2}, we see that necessarily
\begin{equation}
\label{ito-tanaka-3}
\begin{split}
|\check{Y}_t| &= |x| + \int_0^t {\rm sgn}(\check{Y}_s)d\check{W}_s + L_t^{0}(|\check{Y}|) = |x| + \int_0^t {\rm sgn}(\check{Y}_s)d{W}_s + L_t^{0}(|\check{Y}|),
\end{split}
\end{equation}
so that
$$
\int_0^t {\rm sgn}(\check{Y}_s)d\check{W}_s = \int_0^t {\rm sgn}(\check{Y}_s)dW_s,
$$
and
$$
\check{W}_t = \int_0^t {\rm sgn}(\check{Y}_s){\rm sgn}(\check{Y}_s)d\check{W}_s = \int_0^t {\rm sgn}(\check{Y}_s){\rm sgn}(\check{Y}_s)dW_s = W_t.
$$
This ends the proof.

\end{proof}

\begin{remark}
Theorem \ref{convergence} gives a construction of a solution of \eqref{ISBM} when the technical assumption $\mathcal{H}$ is satisfied. When this is not the case, we may conclude that there exists a weak solution to \eqref{ISBM}, with the help of Proposition \ref{prop-borel}. 
\end{remark}

\begin{remark}
We have not been able to prove directly that
\begin{equation*} 
\lim_{n\rightarrow +\infty}{\P}\left [ \left |\int_0^1{\bf \check{\beta}}_n(s) dL^0_s(B^{{\bf \check{\beta}}_n}) -\int_0^1 {\bf {\beta}}(s)  dL^0_s(\check{Y}) \right |>\varepsilon\right ] = 0. 
\end{equation*}
This convergence is strongly related to the convergence of the sequence $(G_t^{\check{\beta}_n})_{n\geq 0}$ towards $\check{G}_t:= \sup\{0\leq s\leq t~:~\check{Y}_s = 0\}$, which cannot be guaranteed neither by the uniform convergence of $(B^{\check{\beta}_n})_{n\geq 0}$ towards $\check{Y}$, nor by the monotonicity of $(B^{\check{\beta}_n})_{n\geq 0}$.
\end{remark}




\section*{Acknowledgements}
Both authors would like to thank Pr. A. Gloter for helpful discussion.





\bibliographystyle{amsplain}

\providecommand{\bysame}{\leavevmode\hbox to3em{\hrulefill}\thinspace}
\providecommand{\MR}{\relax\ifhmode\unskip\space\fi MR }
\providecommand{\MRhref}[2]{%
  \href{http://www.ams.org/mathscinet-getitem?mr=#1}{#2}
}
\providecommand{\href}[2]{#2}

\end{document}
